\documentclass[reqno]{amsart}
\usepackage{enumerate}
\usepackage{amsfonts,amssymb,amsmath,amsthm}
\usepackage{epsfig}
\usepackage{graphics}
\usepackage[normalem]{ulem}
\usepackage{color}
\usepackage{version}
\usepackage{caption}
\usepackage{subcaption}
\usepackage{tikz}
\usepackage{hyperref}

\input xy 
\xyoption{all}
\numberwithin{equation}{section}

\definecolor{OrangeRed}{cmyk}{0,0.6,1,0}            
\definecolor{DarkBlue}{cmyk}{1,1,0,0.20}
\definecolor{DarkGreen}{cmyk}{1,0,0.6,0.2}
\definecolor{myblue}{rgb}{0.66,0.78,1.00}
\definecolor{Violet}{cmyk}{0.79,0.88,0,0}
\definecolor{Lavender}{cmyk}{0,0.48,0,0}

\newtheorem{thm}{Theorem}[section]

\newtheorem{main theorem}[thm]{Main Theorem}
\newtheorem{corollary}[thm]{Corollary}

\newtheorem{lemma}[thm]{Lemma}

\theoremstyle{definition}

\newtheorem{definition}[thm]{Definition}
\newtheorem{remark}[thm]{Remark}

\def\bcases{\begin{cases}}

\def\ecases{\end{cases}}

\newcommand{\bea}{\begin{eqnarray*}}
\newcommand{\eea}{\end{eqnarray*}}

\newcommand{\be}{\begin{equation}}
\newcommand{\ee}{\end{equation}}

\renewcommand{\Re}{\mathrm{Re}\,}

\renewcommand{\epsilon}{\varepsilon}
\renewcommand{\phi}{\varphi}

\newcommand{\Vhat}{\widehat{V}}

\DeclareMathOperator{\Int}{int}

\begin{document}

\title[Limit of the zero locus of the independence polynomial]{The limit of the zero locus of the Independence Polynomial for bounded degree graphs}

\author[F. Bencs]{Ferenc Bencs}
\author[P. Buys]{Pjotr Buys}
\author[H. Peters]{Han Peters}

\today

\address{F. Bencs: Korteweg de Vries Institute for Mathematics\\
University of Amsterdam\\
the Netherlands}
\address{P. Buys: Korteweg de Vries Institute for Mathematics\\
University of Amsterdam\\
the Netherlands}
\address{H. Peters: Korteweg de Vries Institute for Mathematics\\
University of Amsterdam\\
the Netherlands}

\begin{abstract}
The goal of this paper is to accurately describe the maximal zero-free region of the independence polynomial for graphs of bounded degree, for large degree bounds. In previous work with de Boer, Guerini and Regts it was demonstrated that this zero-free region coincides with the normality region of the related occupation ratios. These ratios form a discrete semi-group that is in a certain sense generated by finitely many rational maps. We will show that as the degree bound converges to infinity, the properly rescaled normality regions converge to a limit domain, which can be described as the maximal boundedness component of a semi-group generated by infinitely many exponential maps.

We prove that away from the real axis, this boundedness component avoids a neighborhood of the boundary of the limit cardioid, answering a recent question by Andreas Galanis. We also give an exact formula for the boundary of the boundedness component near the positive real boundary point. 

\end{abstract}

\maketitle

\section{Introduction}

Zeros of partition functions play an important role in quite different contexts. In statistical physics zero parameters can be used to describe phase transitions, as was shown in classical results of Lee and Yang~\cite{LYI52}. Indeed, a limit pressure function is analytic at a positive real (physical) parameter when there is a zero-free neighborhood of the parameter in the complex plane.

A different motivation arises in computational complexity theory, when describing parameters for which there exist efficient algorithms for the approximation of partition functions. It was recently shown by Patel and Regts~\cite{PaR17}, building upon work of Barvinok~\cite{Barbook}, that such algorithms exist for parameters in the zero-free domain containing the origin. In contrast to the result of Lee and Yang, the existence of an efficient algorithm holds for more general families of graphs, and for complex parameter values.

The character of zero-free domains depends not only on the family of graphs, but of course also on the partition function under consideration. In this article we will consider zeros of the independence polynomial, a topic that has seen large progress in recent years. 

The independence polynomial of a graph $G$ is the generating function of the sizes of the independent sets in $G$, i.e. 
\[
  		Z_G(\lambda) = \hspace{-20 pt} \sum_{\substack{I \subseteq V: \\ \text{$I$ is independent}}} 
				\hspace{-20 pt} \lambda^{|I|}.
\]
We note that the independence polynomial arises in statistical physics as the partition function of the hard-core model.

In this paper we will consider the zero sets of the independence polynomial for the class of all graphs of some bounded degree, and study the asymptotic behavior as the degree bound converges to infinity.

We will denote by $\mathcal{G}_d$ the class of all graphs of maximal degree at most $d+1$. We let 
$\mathcal{Z}_d$ denote the set of complex values that can occur as 
a zero of the independence polynomial of such graphs. Furthermore, we
let $\mathcal{U}_d$ denote the complement of the closure of $\mathcal{Z}_d$,
i.e. the maximal open set that is zero-free for all graphs with maximum 
degree at most $d+1$. We denote the connected component of $\mathcal{U}_d$
containing $0$ by $U_d(0)$.

A related set is given by $\mathcal{C}_d$, the maximal zero-free domain containing $0$ for the set of all $d$-ary trees of constant depth. Since those trees are contained in $\mathcal{G}_d$, it follows immediately that $\mathcal{U}_d(0) \subseteq \mathcal{C}_d$. It was shown in \cite{Galanisetal20} that approximation of the independence polynomial for graphs in $\mathcal{G}_d$ is $\#$P-hard for parameters in $\mathbb Q[i] \setminus (\mathcal{C}_d \cup \mathbb{R}_{\geq 0})$.
More recently in~\cite{BBGPR} it was shown that zeros of graphs $G \in \mathcal{G}_d$ are dense outside of $\mathcal{C}_d$, i.e. $\mathcal{U}_d \subseteq \mathcal{C}_d$. It was also shown that parameters
for which approximation of the independence polynomial is $\#$P-hard are dense in $\mathbb C \setminus \mathcal{U}_d$.

The domain $\mathcal{C}_d$ can be parameterized by
$$
\mathcal{C}_d = \left\{\frac{- d^d u}{(u + d)^{d+1}} : |u| <1\right\},
$$
see~\cite{PR19}. The intersection of $\mathcal{C}_d$ with the real axis therefore equals the interval $\left(\frac{-d^d}{(d+1)^{d+1}},\frac{d^d}{(d-1)^{d+1}}\right)$. It was shown in \cite{PR19} that this real interval is contained in $\mathcal{U}_d$, a fact conjectured earlier by Sokal~\cite{Scott_2005}. Thus, the sets $\mathbb R \cap \mathcal{U}_d$ and $\mathbb R \cap \mathcal{C}_d$ are equal, a statement that can be interpreted as an extremality result for the class of $d$-ary trees of constant depth within the graphs of maximal degree at most $d+1$.

Naturally it was asked whether equality of the sets $\mathcal{U}_d$ and $\mathcal{C}_d$ holds in the complex domain as well. It was shown in \cite{Buys21} that this is not the case: there are graphs in $\mathcal{G}_d$  for all $2 \leq d \le 8$ with zeros inside $\mathcal{C}_d$. While it is likely that the methods from \cite{Buys21} extend to arbitrarily high degrees, the proof involves computer calculations which can only run for relatively small degrees. The results in our current paper imply that the negative answer from \cite{Buys21} holds for large degrees $d$ as well.

\medskip

A second natural question asks whether the zero-parameters inside $\mathcal{C}_d$ converge to the boundary of the (properly rescaled) sets $\mathcal{C}_d$ as $d\rightarrow \infty$, thus disappearing in the limit. We will see that this is not the case. The sets $\mathcal{C}_d$ shrink down to the origin as $d \rightarrow \infty$, hence the sets $\mathcal{U}_d \subset \mathcal{C}_d$ do too. However, the rescaled sets $d \cdot \mathcal{C}_d$ converge (in terms of the Hausdorff distance) to a limit set $\mathcal{C}_\infty$, given by
\begin{equation}
    \label{eq: Cinfty}
\mathcal{C}_\infty = \{ -u e^{-u}\; : \; |u| <1\},
\end{equation}
see the illustration in Figure \ref{fig:cardioid}. This ``infinite-degree cardioid'' intersects the real axis in the open interval $(-e^{-1}, e)$.

In this paper we describe the limit of the scaled zero-free regions
$d \cdot \mathcal{U}_d$. We define a complex domain 
$\mathcal{U}_\infty$ (see Definition~\ref{def: Uinfty}), show
that it is the limit of the sets $d \cdot \mathcal{U}_d$, and determine some geometric properties of $\mathcal{U}_\infty$.
Our main results are summarized in the following two theorems.  
The locations of their proofs within the paper can be found at the end of the introduction.

\begin{thm}
    \label{thm: theorem A}
    The sets $d \cdot \mathcal{U}_d$ converge to
    $\mathcal{U}_\infty$ in terms of the Hausdorff distance.
\end{thm}

In other words, for any closed $K_1 \subseteq \mathcal{U}_\infty$ and any open $K_2 \supseteq \overline{\mathcal{U}_\infty}$ there exists a $d_0$ such that $K_1 \subseteq d \cdot \mathcal{U}_d \subseteq K_2$ for $d \geq d_0$.

\begin{thm}
    \label{thm: theorem B}
    \begin{enumerate}
        \item 
        \label{item: item star convex}
        The set $\mathcal{U}_\infty$ is star-convex from $0$, i.e. if 
        $\Lambda \in \mathcal{U}_\infty$ then $c \cdot \Lambda \in \mathcal{U}_\infty$ for all $c \in [0,1]$.
        \item
        \label{item: item (a) theorem B}
        The boundaries of $\mathcal{U}_\infty$ and $\mathcal{C}_\infty$ intersect only in the two real parameters $-e^{-1}$ and $e$. In particular, the smaller set $\mathcal{U}_\infty$ avoids a neighborhood of any non-real boundary point of $\mathcal{C}_\infty$.
        \item Near the positive real boundary point $e$
        \label{item: gamma curve}
        the boundary of $\mathcal{U}_\infty$ is contained in an analytic curve $\Gamma$ (given explicitly in Corollary~\ref{cor: curve Gamma}) and its complex conjugate $\overline{\Gamma}$, and is in particular piece-wise analytic, but not smooth. 
    \end{enumerate}
\end{thm}

\begin{figure}[ht]
\centering
\includegraphics[width=0.6\textwidth]{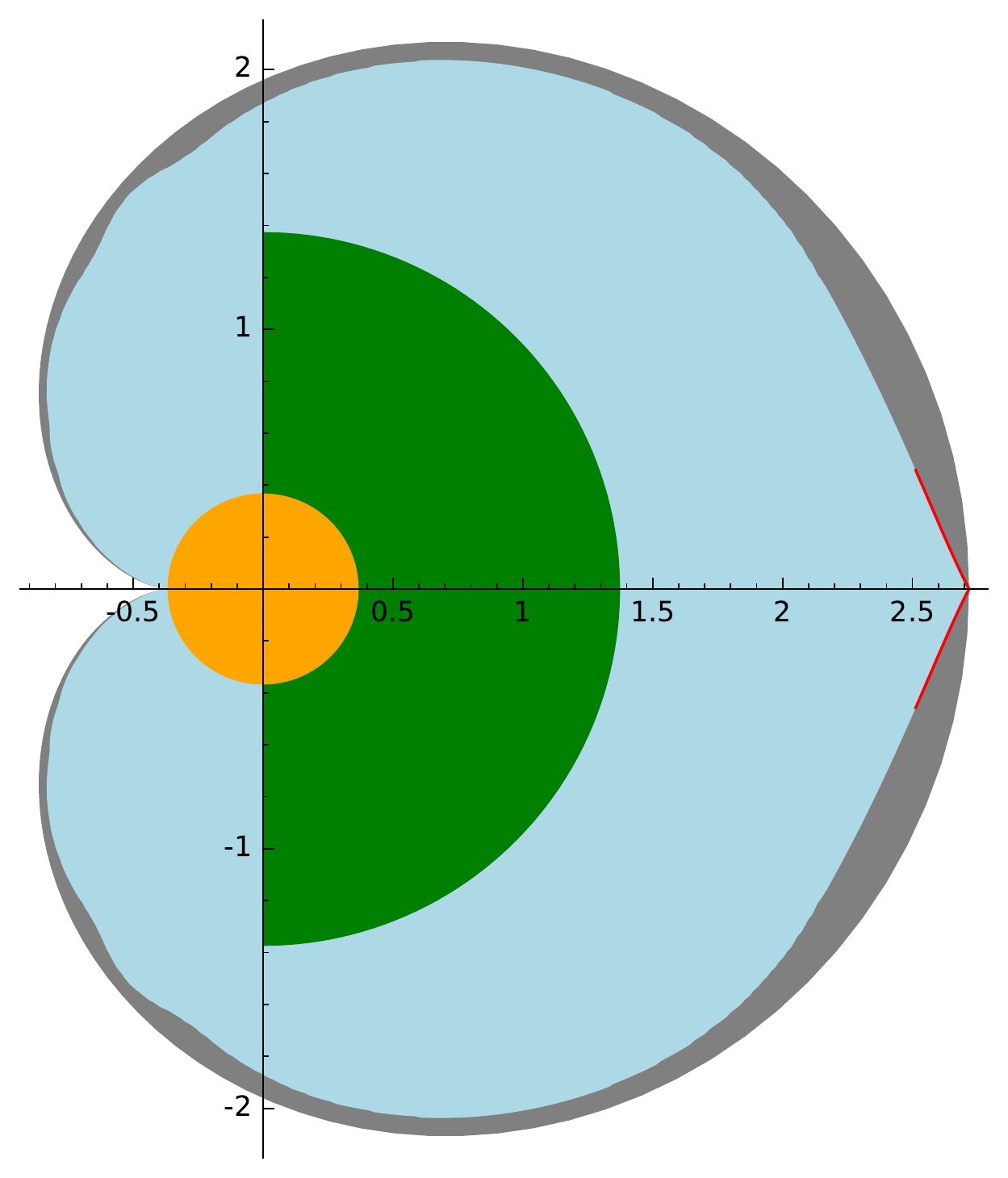}
\caption{The cardioid $\mathcal{C}_\infty$, with known zero-free regions depicted in yellow and green, and the curve $\Gamma$ depicted in red. The gray strip 
 near the boundary of the cardioid is known not to be zero-free anywhere, i.e.\ it does not intersect $\mathcal{U}_\infty$. The algorithm used to compute the gray region relies upon Lemma \ref{lem: 0 on the boundary}.}
\label{fig:cardioid}
\end{figure}

As an immediate consequence of Theorems \ref{thm: theorem A} and \ref{thm: theorem B} we obtain the following:

\begin{corollary}
For all sufficiently large degrees $d\in \mathbb N$ the set $\mathcal{U}_d$ is strictly smaller than $\mathcal{C}_d$.
\end{corollary}

Considering the result from \cite{Buys21} we expect that $\mathcal{U}_d$ and $\mathcal{C}_d$ are unequal for any $d \ge 2$. Theorem~\ref{thm: theorem B} in fact shows that as $d \rightarrow \infty$, the difference between the rescaled seta $d\cdot \mathcal{U}_d$ and $d\cdot \mathcal{C}_d$ does not vanish, which answers a question of Andreas Galanis posed at a recent conference.

\subsection{Known subsets of $\mathcal{U}_d$.}

There is an extensive history of estimates on the zero-free regions for
the independence polynomial for different classes of graphs. We discuss some of the most
relevant results here, focusing only on the class of graphs $\mathcal{G}_d$ and
the region $\mathcal{U}_d$.

\begin{thm}[\cite{Shearer_1985,Scott_2005}]
	For any $d \ge 2$ the disk centered at $0$ with radius $d^d/(d+1)^{(d+1)}$
	is contained in $\mathcal{U}_d$.
\end{thm}

Since zeros accumulate on the point $-d^d/(d+1)^{(d+1)}$, the radius of this so-called
Shearer disk is sharp.

\begin{thm}[\cite{bencs2018note}]
	For any $d \ge 2$ the semi-disk given by the intersection of the disk of
	radius $\frac{7}{8}\tan\left(\frac{\pi}{2d}\right)$ centered at $0$ with
	the right half plane is contained in $\mathcal{U}_d$.
\end{thm}

Due to Theorem \ref{thm: theorem A} we immediately obtain the following:

\begin{corollary}
	The Shearer disk $D_{e^{-1}}(0)$ and the semi-disk $D_{7\pi/16}(0)\cap\{Z:\textrm{Re}(Z)>0\}$
	are contained in $\mathcal{U}_\infty$.
\end{corollary}

Since $\mathcal{U}_\infty$ is contained in $\mathcal{C}_\infty$, the radius of the Shearer disk $D_{e^{-1}}(0)$ is sharp.

\begin{thm}[\cite{PR19}]
	For any $d \ge 2$ the set $\mathcal{U}_d$ contains an open neighborhood
	of the real interval $[0,\frac{d^d}{(d-1)^{d+1}})$.
\end{thm}

As this statement does not give bounds on the size of the neighborhood of the real interval that
are uniform after rescaling, we cannot deduce directly from Theorem \ref{thm: theorem A} that $\mathcal{U}_\infty$
contains an open neighborhood of the real interval $[0,e)$. However, we can prove this result directly:

\begin{thm}
	The set $\mathcal{U}_\infty$ contains an open neighborhood of the real interval $[0,e)$.
\end{thm}

In fact, parts~{(\ref{item: item star convex})} and {(\ref{item: gamma curve})} of Theorem~\ref{thm: theorem B} together imply that $\mathcal{U}_\infty$ contains a neighborhood of $[0,e)$.

\begin{subsection}{Organization}
The organization of this paper is the following. In section~\ref{sec: preliminaries} we will cover background material and discuss known results regarding the sets $\mathcal{U}_d$; we moreover define the set $\mathcal{V}_\infty$ and its interior $\mathcal{U}_\infty$. In section~\ref{sec: properties} we prove properties of these sets. Our main results rely on the analysis done in this section.
Section~\ref{sec: The main proof} is dedicated to the proof of Theorem~\ref{thm: theorem A}, and section~\ref{sec: gamma curve} to the proof of part~{(\ref{item: gamma curve})} of Theorem~\ref{thm: theorem B}.

Specifically: parts~{(\ref{item: item star convex})}, {(\ref{item: item (a) theorem B})} and {(\ref{item: gamma curve})} of Theorem~\ref{thm: theorem B} are proved in Lemma~\ref{lem: star convex}, Corollary~\ref{cor:thm B} and Corollary~\ref{cor: curve Gamma}. Theorem~\ref{thm: theorem A} is proved in section~\ref{sec: the proof}.
\end{subsection}

\section{Preliminaries}
\label{sec: preliminaries}
\subsection{From finite graphs to rational dynamics}
 
In this section we will recall how the set of zeros of bounded degree 
graphs can be related to a semi-group generated by rational maps. Most results in this section will be stated without proofs. The proofs can be found in \cite{BBGPR} or the references therein.

\begin{definition}
	Given a rooted graph $(G,v)$, write
	\[
		Z^{in}_{G,v}(\lambda) = \sum_{I:\, v\in I} \lambda^{|I|}
		\quad
		\text{ and }
		\quad
		Z^{out}_{G,v}(\lambda) = \sum_{I:\, v\notin I} \lambda^{|I|},
	\]
	where both summations run over the independent vertex subsets $I \subset V(G)$.
	We moreover define the \emph{occupation ration} of $G$ at $v$ as the rational function
	\[
		R_{G,v}(\lambda) = \frac{Z^{in}_{G,v}(\lambda)}{Z^{out}_{G,v}(\lambda)}.
	\]
\end{definition}

Since $Z_G(\lambda) = Z_{G,v}^{in}(\lambda) + Z_{G,v}^{out}(\lambda)$, an immediate consequence is that $Z_G(\lambda) = -1$ if and only either $R_{G,v}(\lambda) = 0$ or both $Z_{G,v}^{in}$ and $Z_{G,v}^{out}$ vanish at $\lambda$. While it is in fact possible that both $Z_{G,v}^{in}$ and $Z_{G,v}^{out}$ vanish simultaneously, we can get rid of this condition by considering the whole family of graphs $G \in \mathcal{G}_d$. To make this explicit we introduce some notation. 

Let $\mathcal{G}_d^{k}$ denote the set of rooted graphs $(G,v)$ such that the maximum degree of $G$ is at most $d+1$ and the degree of the vertex $v$ is at most $k$. We let
\[
    \mathcal{R}_d^k = \{R_{G,v} : (G,v) \in \mathcal{G}_d^k \}
\]
denote the corresponding family of occupation ratios. For any $\lambda$ in the Riemann sphere $\widehat{\mathbb{C}}$ we denote by $\mathcal{R}_d^k(\lambda)$ the set of all rational functions in $\mathcal{R}_d^k$ evaluated at $\lambda$. Thus $\mathcal{R}_d^k$ is a family of rational functions, and its values $\mathcal{R}_d^k(\lambda)$ are contained in $\widehat{\mathbb{C}}$. Recall that $\mathcal{Z}_d$ denotes the set of all zeros for graphs of degree at most $d+1$.

\begin{lemma}[Lemma 2.3.~in \cite{BBGPR}]
    \label{lem: -1}
    Let $d\geq 2$, then $\lambda \in \mathcal{Z}_d$ if and only if 
    $-1 \in \mathcal{R}_d^d(\lambda)$.
\end{lemma}




It was shown in \cite{Scott_2005,Weitz_2006,Bencs18} 
that given a rooted graph $(G,v) \in \mathcal{G}_d^k$ there is a rooted tree $(T,u) \in \mathcal{G}_d^k$ such that $R_{G,v} = R_{T,u}$. This is especially useful because for a rooted tree $(T,u)$ the ratio can be expressed recursively. Given $k \ge 1$ we define
$$
F_\lambda(z_1, \ldots, z_k) = \frac{\lambda}{ \prod_{j=1}^k (1 + z_j)}.
$$
Note that the map $F_\lambda$ is defined for any number of inputs.


\begin{definition}
	Let $d \ge 2$ and $\lambda \in \mathbb{C}$. We recursively define a class of rational functions $\mathcal{F}_{d,\lambda}$ as follows:
	\begin{enumerate}
		\item
		the identity map $z \mapsto z$ is contained in $\mathcal{F}_{d,\lambda}$;
		\item
		given any $1 \le k \le d$ and functions $f_1, \ldots , f_k \in \mathcal{F}_{d,\lambda}$, the rational function
		\[
			z \mapsto F_\lambda(f_1(z), \ldots , f_k(z))
		\]
		is also contained in $\mathcal{F}_{d,\lambda}$.
	\end{enumerate}
\end{definition}

The following lemma follows for example from Lemma~2.4.~in \cite{BBGPR} and gives the relation between the rational semi-group $\mathcal{F}_{d,\lambda}$ parameterized by $\lambda$ and the set of occupation ratios.

\begin{lemma}
    \label{lem: ratiofunctions}
    Let $d\geq 2$ and $f_\lambda \in \mathcal{F}_{d,\lambda}$ not equal to the identity. Write
    \[
        f_{\lambda}(z) = F_{\lambda}(f_{1}(z), \dots, f_{k}(z)).
    \]
    The rational map $\lambda \mapsto f_{\lambda}(0)$ is an element of $\mathcal{R}_d^k$, moreover the set of rational functions 
    \[
        \{\lambda \mapsto f_\lambda(0): f_{\lambda} \in 
        \mathcal{F}_{d,\lambda}
        \}
    \]
    is equal to $\mathcal{R}_d^d$.
\end{lemma}

\begin{definition}
	Let $d \ge 2$ and fix $\lambda \in \mathbb C$. We say that a set $Y \subset \mathbb C \setminus \{-1\}$ is invariant under the
	maps $F_\lambda$ if for every $1\le k \le d$ and every $z_1, \ldots, z_k \in Y$ we have
	$$
		F_\lambda(z_1 , \ldots , z_k) \in Y.
	$$
\end{definition}

A direct consequence of lemmas~\ref{lem: -1}~and~\ref{lem: ratiofunctions} is the following:

\begin{lemma}\label{lemma:invariance}
	Let $d \ge 2$ and $\lambda \in \mathbb C$.
	Suppose there exists a set $Y \subset \mathbb C \setminus \{-1\}$ containing $0$ which is invariant under the maps $F_\lambda$. Then $Z_G(\lambda) \neq 0$ for all graphs $G \in \mathcal{G}_d$.
\end{lemma}

An important technique used in many recent papers in the area is the construction of new graphs by \emph{implementing} one rooted graph into another graph, i.e. by replacing each vertex of the latter graph by a copy of the former, where edges now connect corresponding copies the root vertices. This technique adds the maximal degrees of the latter graph to the degree of the root vertex of the former,  hence the sum of these two numbers should be at most $d+1$ to remain in the class $\mathcal{G}_d$. 

\begin{lemma}
    \label{lem: implementation}
    Let $d,k,m \in \mathbb{Z}_{\geq 1}$ such that $k+m\le d$. If there is a rooted graph
    $(G,v) \in \mathcal{G}_{d}^k$ and a $\lambda_0 \in \widehat{\mathbb{C}}$ such that $R_{G,v}(\lambda_0) \in \overline{\mathcal{Z}_m}$, then $\lambda_0 \in \overline{\mathcal{Z}_d}$.  
\end{lemma}

\begin{proof}
    Suppose $(H,u) \in \mathcal{G}_m^m$. We transform the graph $H$ to $\tilde{H}$ by replacing each of its vertices by a copy of the rooted graph $(G,v)$. On the level of ratios this has the effect of composition, that is $R_{\tilde{H},\tilde{u}}(\lambda) = R_{H,u}(R_{G,v}(\lambda))$, where $\tilde{u}$ in $\tilde{H}$ is the vertex corresponding $u$ in the original graph $H$. A proof of this fact can be found for example in Lemma~{2.8} of \cite{BBGPR}. Note that the maximum degree in $\tilde{H}$ is at most $\max{(m+1+k,d+1)} = d+1$ and the degree of $\tilde{u}$ is  at most $m+k \leq d$ and thus $(\tilde{H},\tilde{u}) \in \mathcal{G}_d^d$. This is true for any graph $(H,u) \in \mathcal{G}_m^m$ and therefore we find that $\mathcal{R}_m^m(R_{G,v}(\lambda)) \subseteq \mathcal{R}_d^d(\lambda)$ for every $\lambda \in \widehat{\mathbb{C}}$.
    
    Because $R_{G,v}$ is an open map and $R_{G,v}(\lambda_0) \in \overline{\mathcal{Z}_m}$ we can find a sequence $\{\lambda_n\}_{n \geq 1}$ converging to $\lambda_0$ such that $R_{G,v}(\lambda_n) \in \mathcal{Z}_m$ for all $n$. According to Lemma~\ref{lem: -1} this means that $-1$ is contained in $\mathcal{R}_m^m(R_{G,v}(\lambda_n))$ for all $n$. This thus means that $-1 \in \mathcal{R}_d^d(\lambda_n)$ for all $n$, which again by Lemma~\ref{lem: -1} implies that $\lambda_n \in \mathcal{Z}_d$. We conclude that $\lambda_0 \in \overline{\mathcal{Z}_d}$.
\end{proof}

\subsection{From rational to transcendental dynamics}

In this section we will define the ``rescaled limit'' of objects from the previous subsection. We start by formalizing the relationship between the cardioids $\mathcal{C}_d$ and their limit $\mathcal{C}_\infty$.

\begin{remark}\label{rem:cardioid}
The cardioid $d\cdot \mathcal{C}_d$ contains all parameters $\Lambda$ for which the function $Z\mapsto \Lambda/(1+Z/d)^d$ has an attracting fixed point. As $d \rightarrow \infty$ these rational functions converge locally uniformly to the entire functions $E_\Lambda(Z) = \Lambda \cdot e^{-Z}$, and the domains $d\cdot C_d$ converge in the Hausdorff distance to the domain $\mathcal{C}_\infty$. It follows that the set $\mathcal{C}_\infty$ coincides with the hyperbolic component of the family $E_\Lambda$, containing all parameters $\Lambda$ for which $E_\Lambda$ has an attracting fixed point.

We note that this is a particular case of the more general phenomenon described in~\cite{KK2001}. In our setting the relationship between the hyperbolic components can also be computed explicitly as follows:

Suppose $E_{\Lambda}$ has a fixed point $p$ with derivative
$E_{\Lambda}'(p) = \alpha$. Then, because 
$E_{\Lambda}(Z)/E_{\Lambda}'(Z)=-1$ we find that 
$p=-\alpha$ and thus $\Lambda = -\alpha e^{-\alpha}$.
In the other direction, if $\Lambda = -\alpha e^{-\alpha}$,
then $-\alpha$ is a fixed point of $E_{\Lambda}$ with derivative $\alpha$.
\end{remark}


For any $n \in \mathbb N$ and $s_1, \ldots , s_n \ge 0$ with $\sum_{i=1}^n s_i \le 1$, and for fixed $\Lambda \in \mathbb C$, define the function $E_{\Lambda,(s_1, \ldots, s_n)}: \mathbb C^n \rightarrow \mathbb C$ by
$$
E_{\Lambda,(s_1, \ldots, s_n)}(Z_1, \ldots , Z_n) = \Lambda \cdot e^{-s_1 Z_1 - \cdots - s_n Z_n}.
$$
We will denote $E_{\Lambda,(1)}$ by $E_{\Lambda}$.
We define the set of functions $G_\Lambda$ in the following way.

\begin{definition}
    Let $\Lambda \in \mathbb{C}$. We recursively define a class of transcendental functions $G_\Lambda$ as follows:
	\begin{enumerate}
		\item
		the identity map $Z \mapsto Z$ is contained in $G_\Lambda$;
		\item
		for any $k \geq 1$, tuple $(s_1, \dots, s_k) \in \mathbb{R}_{\geq 0}^k$ with $\sum_{m=1}^k s_m \leq 1$ 
		and functions $g_1, \dots, g_k \in G_\Lambda$ the function
		\[
			Z \mapsto E_{\Lambda,(s_1, \dots, s_k)}\left(g_1(Z), \dots, g_k(Z)\right)
		\]
		is also contained in $G_\Lambda$.
	\end{enumerate}
\end{definition}

We define the set $V_\Lambda$ as the orbit of $0$ under 
the set of functions $G_\Lambda$, 
i.e. $V_{\Lambda} = \{g(0): g \in G_{\Lambda}\}$. 
We define $\Vhat_{\Lambda}$ as the convex hull of $V_{\Lambda}$.

The following lemma follows directly from the definitions, but gathers some elementary properties of $G_{\Lambda}$,
$V_\Lambda$ and $\Vhat_{\Lambda}$ that will be used a lot.
\begin{lemma}  \label{lem: basic properties}
    Let $\Lambda \in \mathbb{C}$.
    \begin{enumerate}
        \item We have $E_{\Lambda}(\Vhat_\Lambda) = V_{\Lambda}\setminus\{0\}$.
        \item If $K$ is a convex set containing $0$ that is
        forward invariant for $E_{\Lambda}$, then it is 
        forward invariant for every map in $G_{\Lambda}$ and 
        $\Vhat_{\Lambda} \subseteq K$.
    \end{enumerate}
\end{lemma}

\begin{definition}
    \label{def: Uinfty}
    We define the set $\mathcal{V}_\infty$ as the set consisting of those $\Lambda$ for which $V_\Lambda$ is bounded. We define 
    $\mathcal{U}_\infty$ as its interior.
\end{definition}

\begin{lemma}
    \label{lem: star convex}
    Part~{(\ref{item: item star convex})} of Theorem \ref{thm: theorem B} holds, i.e. the set $\mathcal{U}_\infty$ is star-convex from $0$.
\end{lemma}

\begin{proof}
    We will first show that $\mathcal{V}_\infty$ is star-convex from $0$. So we suppose that $\Lambda \in \mathcal{V}_{\infty}$ and we will argue that $t \cdot \Lambda \in \mathcal{V}_{\infty}$ for all $t \in [0,1]$. By definition $\Vhat_{\Lambda}$ is a bounded convex set containing $0$ that is forward invariant for $E_{\Lambda}$. We find that 
    \[
        E_{t \cdot \Lambda}(\Vhat_{\Lambda}) = t \cdot E_{\Lambda}(\Vhat_{\Lambda}) \subseteq t \cdot \Vhat_{\Lambda} \subseteq \Vhat_{\Lambda}
    \]
    and thus it follows from Lemma~\ref{lem: basic properties} that $\Vhat_{t \cdot \Lambda} \subseteq \Vhat_{\Lambda}$. Therefore $V_{t \cdot \Lambda}$ is bounded and thus $t \cdot \Lambda \in \mathcal{V}_\infty$. We conclude that $\mathcal{V}_\infty$ is star-convex from $0$. 
    
    To conclude the same for its interior $\mathcal{U}_\infty$ we only need to show that $0 \in \mathcal{U}_\infty$. This follows from the fact that if $|\Lambda| < 1/e$ the unit disk $\mathbb{D}$ is forward invariant for $E_{\Lambda}$. Therefore it follows from Lemma~\ref{lem: basic properties} that $V_{\Lambda} \subseteq \mathbb{D}$ and thus $\Lambda \in \mathcal{V}_\infty$. This shows that the open disk of radius $1/e$ is contained in $\mathcal{U}_\infty$.
\end{proof}

The next three lemmas form the bridge between the finite (but large degree) world and the transcendental one. First let us emphasize the relation between $F_\lambda$ and $E_{\Lambda}$.

\begin{lemma}
For each pair of compact subsets $K, L \subset \mathbb C$, $s_1,\dots,s_j> 0$ with $\sum s_i\le 1$, and each $\epsilon >0$, there exists a $d_0\in \mathbb{N}$ such that the following holds: For any $d \ge d_0$, and for any $Z_1, \ldots, Z_j \in K$ and $\Lambda \in L$ we have
$$
\left|d \cdot F_{\Lambda/d}\left(\underbrace{Z_1/d,\ldots,Z_1/d}_{p_1}, \ldots, \underbrace{Z_j/d,\ldots,Z_j/d}_{p_j}\right) - E_{\Lambda,(s_1,\ldots,s_j)}(Z_1, \ldots, Z_j)\right| < \epsilon,
$$
where $p_i=\lfloor s_i d\rfloor$ for $1\le i\le j$.
\end{lemma}
\begin{proof}
    Observe that
    \[
        d \cdot F_{\Lambda/d}\left(\underbrace{Z_1/d,\ldots,Z_1/d}_{p_1}, \ldots, \underbrace{Z_j/d,\ldots,Z_j/d}_{p_j}\right)=\Lambda\prod_{i=1}^j (1+Z_i/d)^{-p_i}.
    \]
    As $d$ tends to infinity, the map $Z_i\mapsto (1+Z_i/d)^{-p_i}$ uniformly converges to $Z_i \mapsto e^{-s_i Z_i}$ on $K$ for all $1\le i\le j$, thus their product also converges uniformly. The statement now follows from the fact that the set $L$ is bounded. 
\end{proof}

\begin{lemma}
    \label{lem: approximate ratios}
    For any $g_\Lambda\in G_\Lambda$ there exists a sequence of graphs $(G_d,v_d)\in\mathcal{G}_d^d$, such that $\Lambda\mapsto d \cdot R_{G_d,v_d}(\Lambda/d)$ converges locally uniformly to $\Lambda\mapsto g_\Lambda(0)$. Furthermore, if we can write $g_\Lambda(Z)=E_{\Lambda,(s_1,\dots,s_j)}(g_1(Z),\dots,g_j(Z))$ then we can choose the sequence such that 
    \[
        \deg(v_d) \leq d\cdot \sum_{1\le i\le j}s_i.
    \]
    
\end{lemma}

\begin{proof}
    The proof relies upon the recursive definition of the family $G_\Lambda$, and uses induction on the number of compositions in the definition of the map $g_\lambda$. The statement is trivial for $Z\mapsto Z$. Let us assume that $g_\Lambda\neq \textrm{id}$. Then we can write $g_\Lambda(Z)=E_{\Lambda,(s_1,\dots,s_j)}(g_{1,\Lambda}(Z),\dots,g_{j,\Lambda}(Z))$ and by induction for each $g_i$ there exists a sequence $f^{(i)}_{\lambda,d}\in\mathcal{F}_{d,\lambda}$, such that 
    \[
        \Lambda\mapsto d\cdot f^{(i)}_{\Lambda/d,d}(0)
    \]
    converges locally uniformly to $\Lambda\mapsto g_{i,\Lambda}(0)$ for each $1\le i\le j$. By the previous lemma we know that the composition
    \[
        \Lambda\mapsto d\cdot F_{\Lambda/d}\left( \underbrace{f_{\Lambda/d,d}^{(1)}(0), \dots, f_{\Lambda/d,d}^{(1)}(0) }_{p_1},\dots,  \underbrace{f_{\Lambda/d,d}^{(j)}(0), \dots, f_{\Lambda/d,d}^{(j)}(0)}_{p_j}\right)
    \]
    converges locally uniformly to $\Lambda\mapsto g_\Lambda(0)$. Since $p_1+\dots+p_j\le d(s_1+\dots+s_j)$ we can conclude by Lemma~\ref{lem: ratiofunctions} the desired statement.
\end{proof}

\begin{lemma}
    \label{lem: strictly forward invariant}
    Let $K$ be convex bounded set containing $0$, such that $E_{\Lambda_0}(K)$ is relatively compact in $\textrm{int}(K)$. Then there exists  $\varepsilon>0$ and $d_0$, such that  $B_{\varepsilon}(\Lambda_0)\subseteq d \cdot \mathcal{U}_d$ for any $d\ge d_0$.
\end{lemma}

\begin{proof}
For $d$ large we can define a function $\delta_d: K \to \mathbb{C}$ such that $1/(1+Z/d) = e^{-Z/d + \delta_d(Z)}$. Note that $\delta_d$ can be chosen such that $d \cdot \delta_d(Z)$ converges to $0$ uniformly on $K$ as $d$ tends to infinity. Let $1\leq k \leq d$ and let $(z_1, \dots, z_k)$ be an arbitrary $k$-tuple in $K/d$ and write $z_i = Z/d$. We find that 
\[
    d \cdot F_{\Lambda/d}(Z_1/d,\cdots, Z_k/d) = 
    \Lambda \prod_{i=1}^k \frac{1}{(1+Z_i/d)} = \Lambda \prod_{i=1}^k e^{-Z_i/d + \delta_d(Z_i)} = 
    \frac{\Lambda}{\Lambda_0} \cdot E_{\Lambda_0}\left(\frac{1}{d} \sum_{i=1}^k Z_i\right) \cdot e^{\sum_{i=1}^k \delta_d(Z_i)}. 
\]
By choosing $d$ large enough the term $e^{\sum_{i=1}^k \delta_d(Z_i)}$ can be bounded arbitrarily close to $1$ independently of the chosen tuple. Furthermore, by taking $\Lambda$ close to $\Lambda_0$, the same is true for the term $\Lambda/\Lambda_0$. Therefore there exists an $\epsilon > 0$ and a $d_0$ such that $K/d$ is forward invariant under the maps $\mathcal{F}_{\Lambda/d}$ for $\Lambda \in B_{\epsilon}(\Lambda_0)$ and $d \geq d_0$. The statement thus follows from Lemma~\ref{lemma:invariance}.
\end{proof}

\section{The sets $\mathcal{V}_\infty$, $\mathcal{U}_\infty$ and $V_\Lambda$}

\label{sec: properties}

\subsection{Properties of $\mathcal{V}_\infty$, $\mathcal{U}_\infty$ and $V_\Lambda$}

In this section we will mostly be concerned with parameters $\Lambda$ that are not real, but for completeness we first show that 
\[
    \mathcal{V}_\infty \cap \mathbb{R} = [-1/e,\infty).
\]
If $\Lambda \geq 0$ then $E_\Lambda|_\mathbb{R}$ is a decreasing function.
Furthermore $E_{\Lambda}(0) = \Lambda$ and $E_{\Lambda}(\Lambda) > 0$ and thus the (convex) interval $[0,\Lambda]$ is forward invariant for $E_{\Lambda}$. Therefore $\Vhat_{\Lambda}$ is contained in $[0,\Lambda]$ and in fact, because $\Vhat_\Lambda$ must contain both $0$ and $\Lambda$, $\Vhat_\Lambda = [0,\Lambda]$. For $-e^{-1} \leq \Lambda < 0$ the map $E_{\Lambda}$ has an attracting or a neutral fixed point $w_{\Lambda} < 0$. In this case the interval $(w_{\Lambda},0]$ is forwards invariant and thus it contains $\Vhat_{\Lambda}$. Because the orbit of $0$ under iteration of $E_{\Lambda}$ must converge to $w_{\Lambda}$ we see that $\Vhat_{\Lambda} = (w_{\Lambda},0]$. Finally, if $\Lambda < -1/e$ the equation $E_{\Lambda}(x) = x$ has no real solutions and thus the orbit of $0$, which is a decreasing sequence, must converge to $-\infty$. Therefore $\Lambda$ is not contained in $\mathcal{V}_\infty$ in this case.

The statements in our main theorem concern the interior of $\mathcal{V}_\infty$, which we have denoted by $\mathcal{U}_\infty$. In the course of this paper we will show that $\mathcal{V}_\infty$ is equal 
to the disjoint union of the closure of $\mathcal{U}_\infty$ and the interval $(e,\infty)$.

\begin{lemma}
    \label{lem: Vhat is C}
	If $\Lambda \not \in \mathcal{V}_\infty \cup \mathbb{R}$ 
	then $V_{\Lambda} = \mathbb{C}$.
\end{lemma}

\begin{proof}
Take $\Lambda$ to be nonreal and such that $V_\Lambda$ is unbounded. Let us suppose that $\Vhat_\Lambda \neq \mathbb{C}$. 
The set $\Vhat_\Lambda$ is convex and contains 
the origin and thus it must avoid a sector of the form 
\[
	S := \{Z\in \mathbb C \;:\; |Z| > R, \; \mathrm{and} \; \theta_1 < \mathrm{arg}(Z) < \theta_2\}.
\]
Because $E_{\Lambda}(\Vhat_{\Lambda}) = V_\Lambda$ and $V_\Lambda$ is unbounded 
we can choose $Z_0 \in \Vhat_\Lambda$ for which $|e^{-Z_0}|$ is arbitrarily large. 
We claim that we can choose this $Z_0$ such that in addition $|\textrm{Im}(Z_0)|$ is arbitrarily large.
Assuming this claim holds we obtain a contradiction: the curve $t \mapsto \Lambda e^{-tZ_0}$, for $t \in [0,1]$, 
must pass through the sector $S$, however this curve must also be contained in $\Vhat_\Lambda$, 
contradicting our earlier assumption.

It remains to prove the claim. Suppose that there is a sequence of points $\{Z_n\}_{n \geq 1}$ in $\Vhat_\Lambda$ 
with real parts converging to minus infinity, but bounded imaginary parts.
Write $Z_n = a_n + ib_n$ and $\Lambda = x + iy$. Define a sequence $s_n \in [0,1]$ for which $e^{-s_n a_n}$ converges to infinity slow enough such that 
$a_n + e^{-s_n a_n}(|x|+|y|)$ still converges to $-\infty$. Let $\epsilon = |y|/(2(|x|+|y|))$. Because the sequence $s_n$ necessarily converges to $0$ 
and the sequence $b_n$ is bounded we find that 
$\cos(s_n b_n) > 1-\epsilon$ and $|\sin(s_n b_n)| < \epsilon$ for all $n$ 
sufficiently large.
Now let $W_n = Z_n + \Lambda e^{-s_n Z_n}$, then 
\[
    \textrm{Re}(W_n) = a_n + e^{-a_n s_n} (x \cos (b_n s_n)+y \sin (b_n s_n)) \leq 
    a_n + e^{-a_n s_n}(|x|+|y|).
\]
and
\[
    |\textrm{Im}(W_n)| = |b_n + e^{-a_n s_n} (y \cos(b_n s_n) - x \sin(b_n s_n))|
    \geq e^{-a_n s_n} (|y|(1-\epsilon) - |x| \epsilon) - |b_n| = \frac{|y|}{2}e^{-a_n s_n} - |b_n|.
\]
We see that 
$\mathrm{Re}(W_n) \rightarrow \infty$ and $|\mathrm{Im}(W_n)| \rightarrow \infty$. 
Since both $Z_n$ and $\Lambda e^{-s_n Z_n}$ are contained in $\Vhat_\Lambda$, it follows that $W_n/2 \in \Vhat_\Lambda$.
Working with a large enough point $W_n/2$ instead of $Z_0$ we obtain a contradiction as above.
\end{proof}

Let $f: \mathbb{C} \to \mathbb{C}$ be a holomorphic map, then a point $p \in \mathbb{C}$ is periodic with period $n$ if $f^{\circ n}(p) = p$ and $f^{\circ m}(p) \neq p$ for all $1<m<n$. The derivative $\mu=(f^{\circ n})'(p)$ is called the multiplier of $p$ and $p$ is called attracting, neutral or repelling according to whether $|\mu| <1$, $|\mu| =1$ or $|\mu| > 1$ respectively.

\begin{lemma}
    \label{lem: attracting/neutral fixed points}
    Let $\Lambda \in \mathbb{C} \setminus \mathbb{R}$ and
    suppose that $K\neq\mathbb{C}$ 
    is a convex set containing $0$ for
    which $E_{\Lambda}(K) \subseteq K$. Then 
    $\Lambda \in \mathcal{V}_\infty$, and every non-identity $g \in G_{\Lambda}$ has a fixed point $p$ in the 
    closure of $K$ that is either attracting or 
    neutral with multiplier $1$. All other periodic points of $g$ 
    in the closure of $K$ are repelling and lie on the boundary of $K$.
\end{lemma}

\begin{proof}

Because $K$ is convex and contains $0$ it follows from Lemma~\ref{lem: basic properties} that $g(K) \subseteq K$ and $\Vhat_{\Lambda} \subseteq K$. Therefore, because $K \neq \mathbb{C}$, we can use Lemma~\ref{lem: Vhat is C} to conclude that $\Vhat_\Lambda$ is bounded and $\Lambda \in \mathcal{V}_\infty$. 

The set $\Vhat_\Lambda$ contains the convex hull of the spiral given by the image of the interval $[0,1]$ under the map $t \mapsto \Lambda e^{-t\Lambda}$. Because $\Lambda$ is not real, this spiral is not a straight line segment, and thus both $\Vhat_\Lambda$ and $K$ have a non-empty interior. The interior of $K$ is a simply connected proper subset of $\mathbb{C}$ and thus conformally isomorphic to the open unit disk. The Denjoy–Wolff theorem states that either the action of $g$ on $\Int(K)$ is conjugate to a rotation or the orbit of any initial value in $\Int(K)$ under repeated application of $g$ converges to a fixed point in the closure of $K$ (this point might be $\infty$). We will first argue that the former cannot be the case.

Suppose that the action of $g$ is conjugate to a rotation on $\Int(K)$. Then it should also be conjugate to a rotation on the invariant subset $\Int(\Vhat_{\Lambda})$ and in particular it should be a bijection on this subset. But, since $\Vhat_{\Lambda}$ is bounded and $g(\Vhat_{\Lambda}) \subseteq E_{\Lambda}(\Vhat_{\Lambda})$, $g(\Int(\Vhat_{\Lambda}))$ avoids a neighborhood of $0$ and thus $g$ cannot be a bijection. 

We can conclude that every value in $\Int(K)$ must converge to a point $p$ in the closure of $K$. Note that $p$ must be included in the closure of 
$\Vhat_\Lambda$ and thus cannot be equal to $\infty$.
It follows that $\Int(K)$ is contained in an invariant Fatou component $U$ of the map $g$. There are five possible types of invariant Fatou components for transcendental maps. We have excluded $U$ being either a Siegel disk or a Herman ring, since $g$ exhibits attracting behaviour on $U$. The component also cannot be a Baker domain because $p \neq \infty$. The only two remaining cases are that $U$ is the immediate basin for an attracting fixed point or for a petal of a parabolic fixed point with multiplier $1$. There are no other periodic points of $g$ on $U$ and all other periodic points in the boundary of $U$ are repelling. 

\end{proof}

Recall from Remark \ref{rem:cardioid} that $E_{\Lambda}$ has 
an attracting or a neutral fixed point if and only if 
$\Lambda$ lies in the interior or the boundary of 
$\mathcal{C}_\infty$ respectively (see equation~{(\ref{eq: Cinfty})}). It follows from Lemma~\ref{lem: attracting/neutral fixed points} that the nonreal elements of $\mathcal{V}_\infty $ are contained in 
$\overline{\mathcal{C}_\infty}$. Furthermore, because for any 
$\Lambda \in \partial\mathcal{C}_\infty \setminus \mathbb{R}$ 
the neutral fixed point of $E_\Lambda$ does not have 
multiplier $1$, we see that $\mathcal{V}_\infty \cap \partial \mathcal{C}_\infty = \{-e^{-1},e\}$. To prove Part~{(\ref{item: item (a) theorem B})} of Theorem \ref{thm: theorem B} it is thus enough to show that $\mathcal{V}_\infty$ is closed.

\begin{lemma}\label{lemma:closed}
	The set $\mathcal{V}_\infty$ is closed.
\end{lemma}

\begin{proof}
	We will show that the complement is open. To that effect take a $\Lambda_0$ that is not in $\mathcal{V}_\infty$. If $\Lambda_0$ is real then $\Lambda_0$ lies outside
	the interval $[-e^{-1},\infty)$ and thus there is a
	neighborhood of $\Lambda_0$ that does not intersect
	$\overline{\mathcal{C}_\infty}$. This neighborhood can also be chosen not to intersect with $\mathcal{V}_\infty$ because $\mathcal{V}_\infty$ is a subset of
	$ \overline{\mathcal{C}_\infty} \cup [-1/e,\infty)$. We will henceforth 
	assume that $\Lambda_0$ is not real.
	
	Let $p_{\Lambda_0}$ be a repelling periodic point 
	of $E_{\Lambda_0}$.  Write $p_{\Lambda_0}$ as a sum of the following form
	\[
	    p_{\Lambda_0} = \sum_{i=1}^n s_i \cdot z_i,
	\]
	where the $s_i$ lie in $(0,1]$ with their sum being strictly less than one and the $z_i$ are nonzero complex numbers that are not co-linear. According to Lemma~\ref{lem: Vhat is C} there exist $g_{i,\Lambda_0} \in G_{\Lambda_0}$ such that $g_{i,\Lambda_0}(0) = z_i$ for all $i$.
	
	Recall that
	the values $g_{i,\Lambda}(0)$ depend continuously 
	on the parameter $\Lambda$. For $\Lambda \in \mathbb{C}$
	we define 
	\[
	    A_{\Lambda} = \left\{\sum_{i=1}^n t_i \cdot g_{i,\Lambda}(0): t_i \in (0,1) \text{ with }\sum_{i=1}^n t_i < 1\right\}.
	\]
	The set $A_{\Lambda_0}$ is an open neighborhood of 
	$p_{\Lambda_0}$. Using the implicit function theorem 
	we can extend the repelling periodic point $p_{\Lambda_0}$
	to a neighborhood of $\Lambda_0$ on which there exists 
	a holomorphic 
	function $\Lambda \mapsto p_{\Lambda}$ such that 
	$p_{\Lambda}$ is a repelling periodic point of 
	$E_{\Lambda}$ for all $\Lambda$. By continuity it 
	follows that there is a neighborhood $U$ of $\Lambda_0$
	such that $A_{\Lambda}$ is an open set containing 
	$p_{\Lambda}$ for all $\Lambda \in U$. Because 
	$A_{\Lambda} \subseteq \Int(\Vhat_{\Lambda})$ for all 
	$\Lambda \in U$ it follows from Lemma~{\ref{lem: attracting/neutral fixed points}} that these $\Lambda$ 
	cannot lie in $\mathcal{V}_\infty$. This concludes
	the proof.
\end{proof}

\begin{remark}
A shorter proof of Lemma \ref{lemma:closed} can be given using Lemma \ref{lem: 0 on the boundary} below. To see this, suppose that $\Lambda_0 \notin \mathcal{V}_\infty$, in which case $V_{\Lambda_0} = \mathbb C$. Choose $g_{1,\Lambda}, g_{2,\Lambda}, g_{3,\Lambda} \in \mathcal{G}_{\Lambda}$ such that $0$ lies in the interior of the convex hull of the three points $g_{i,\Lambda_0}(0)$. Since each $g_{i,\Lambda}(0)$ varies holomorphically with $\Lambda$, it follows that the same holds for $\Lambda$ sufficiently close to $\Lambda_0$, which by Lemma \ref{lem: 0 on the boundary} implies Lemma \ref{lemma:closed}. In order to avoid a circular argument we have included the longer proof above.
\end{remark}


\begin{corollary} \label{cor:thm B}
	Part~{(\ref{item: item (a) theorem B})} of Theorem \ref{thm: theorem B} holds, i.e.\
	there is a neighborhood of $\partial{C}_{\infty}\setminus \{-e^{-1},e\}$ that does not intersect $\mathcal{V}_\infty$.
\end{corollary}

\begin{lemma}
    \label{lem: multiplier fixed points boundary}
    Let $\Lambda \in \mathbb{C}$, $g\in G_\Lambda$ and $K$ a convex set 
    with non-empty interior such that $g(K) \subseteq K$. Then any 
    periodic point of $g$ on the boundary of $K$ must have a positive real 
    multiplier.
\end{lemma}

\begin{proof}
Let $p \in \partial K$ be a periodic point of $g$ of order $k$, let $I = (p,q]$ be an interval contained in $K$, and let $L$ be a $H$ be a closed half-plane containing $K$ with $p \in  \partial H$. By invariance $g^{k\cdot n}(I) \subset H$ for all $n \in \mathbb N$, which implies that $(g^k)\prime(p)\ge 0$.
\end{proof}

\begin{lemma}\label{lem: 0 on the boundary}
If $\Lambda \in \mathcal{V}_\infty$ then $0\in \partial \Vhat_\Lambda$.
\end{lemma}

\begin{proof}
 As we have seen, the statement is true for real $\Lambda$, so we assume that $\Lambda$ is not real. In that case $\Lambda$ lies in $\mathcal{C}_\infty$. This means that $E_{\Lambda}$ has a unique attracting fixed point, which we will denote by $w$, and all other fixed points are repelling. Moreover the multiplier $E_{\Lambda}'(w)$ is not real. It follows from Lemmas~\ref{lem: attracting/neutral fixed points} and \ref{lem: multiplier fixed points boundary} that $w$ lies in the interior of $\Vhat_{\Lambda}$.
 
 By the Riemann mapping theorem, there exists a comformal bijection $h:\mathbb{D}\to  \mathrm{int} (\widehat{V}_\Lambda)$ with $h(0)=w$. For the sake of a contradiction let us assume that $0=h(q)$ for some $|q|<1$. Without loss of generality we may assume that $q>0$. Let $D_{q}(0)$ denote the open disk of radius $q$ centered around $0$.
We claim that the image  $T=h \left(\overline{D_{q}(0)}\right)$ 
is convex and $E_\Lambda$-invariant.

To see convexity, let us recall the classic result (see e.g.~\cite[$\S 2.5$]{Duren1983}) that a univalent image $f(\mathbb D)$ is convex if and only if for any $z\in \mathbb{D}$:
\begin{equation}
\label{eq: convexity}
    \Re \left(1+\frac{zf''(z)}{f'(z)}\right) > 0.
\end{equation}
This equation is thus satisfied for $f = h$, and therefore must also be satisfied for the function $f_{q}(z)=h(q\cdot z)$. It follows 
that $T$ is convex, being the closure of $f_{q}(\mathbb{D})$.

To see the forward invariance, write $H=h^{-1}\circ E_\Lambda \circ h$. Clearly $H$ maps the disk into itself and
 \[
    H(0)=h^{-1}(E_\Lambda(h(0)))=h^{-1}(E_\Lambda(w))=h^{-1}(w)=0.
 \]
Therefore the Schwarz Lemma gives
$$
H(D_{q}(0)) \subset D_{q}(0),
$$
and thus $E_\Lambda(T) \subset T$. The set $T$ is therefore a forward invariant convex proper subset of $\Vhat_\Lambda$ containing $0$. It follows from Lemma~\ref{lem: basic properties} that $T$ should contain $\Vhat_\Lambda$, which is clearly a contradiction.
\end{proof}

In Figure \ref{fig:cardioid}, the gray region depicting the set $\mathcal{C}_\infty \setminus \mathcal{V}_\infty$ was computed by checking whether the condition $0 \in \partial \Vhat_\Lambda$ was violated for an approximation of the set $\Vhat_\Lambda$.

\begin{corollary}
    \label{cor: normal family}
	The family $\{\Lambda \mapsto g_{\Lambda}(0): g_{\Lambda} \in G_\Lambda\}$ is normal on $\mathcal{U}_\infty$.
\end{corollary}

\begin{proof}
    For any $\Lambda \in \mathcal{U}_\infty$ the set $\Vhat_\Lambda$ cannot contain any point $z$ of the form $k\cdot i$ for $k \geq 2\pi$ because otherwise $\Vhat_\Lambda$ would contain the convex hull of the circle parameterized by $t \mapsto \Lambda \cdot e^{-t\cdot z}$ contradicting Lemma~\ref{lem: 0 on the boundary}. The family in question is thus normal by Montel's theorem.
\end{proof}

\begin{lemma}\label{lem:neutral fixed point exclusion}
    Let $\Lambda_0 \in \mathbb{C}\setminus \mathbb{R}$. Suppose there is a non-identity map $g \in G_{\Lambda_0}$ and a $g$-invariant convex set $K$ containing $0$ whose closure contains a neutral fixed point of $g$. Then $\Lambda_0 \not \in \mathcal{U}_\infty$.
\end{lemma}

\begin{proof}
    Recall that $g$ depends holomorphically on the parameter $\Lambda$. We consider the family of functions $\mathcal{F} = \{\Lambda \mapsto g_\Lambda^{n}(0): n \geq 1\}$. Corollary~\ref{cor: normal family} states that $\mathcal{F}$ is normal on $\mathcal{U}_\infty$. 
    For $|\Lambda| < 1/e$ the image of the unit disk $\mathbb{D}$ under the map $Z \mapsto E_{\Lambda}(Z)$ is relatively compact in $\mathbb{D}$. Because $\mathbb{D}$ is convex and contains $0$ the same is true for $Z \mapsto g_\Lambda(Z)$. Therefore the orbit $\{g_{\Lambda}^n(0)\}_{n \geq 1}$ converges to an attracting fixed point of $g_{\Lambda}$. Because $\mathcal{F}$ is normal on $\mathcal{U}_\infty$ and $\mathcal{U}_\infty$ is connected containing the disk of radius $1/e$ it follows that 
    the orbit $\{g_{\Lambda}^n(0)\}_{n \geq 1}$ must converge to a fixed point $p_{\Lambda}$ of $g_{\Lambda}$ on the whole of $\mathcal{U}_\infty$. Furthermore the map $\Lambda \mapsto p_{\Lambda}$ is holomorphic. 
    
    For every $\Lambda$ the point $p_{\Lambda}$ lies in the closure of $\Vhat_{\Lambda}$. It follows then from from Lemma~\ref{lem: attracting/neutral fixed points} that if $|g_{\Lambda}'(p_{\Lambda})| = 1$
    then $g_{\Lambda}'(p_{\Lambda})$ is exactly equal to $1$. But because
    $\Lambda \mapsto g_{\Lambda}'(p_{\Lambda})$ is holomorphic and thus, if it is 
    not constant, an open map, this cannot occur on $\mathcal{U}_\infty$. 
    Therefore $p_{\Lambda}$ is an attracting fixed point for all $\Lambda \in \mathcal{U}_\infty$.
    
    If $\Lambda_0$ were an element of $\mathcal{U}_\infty$ then $p_{\Lambda_0}$ would be an attracting fixed point of $g_{\Lambda_0}$ contained in $K$. Lemma~\ref{lem: attracting/neutral fixed points} states that $g_{\Lambda_0}$ cannot contain any other attracting or neutral fixed points in the closure of $K$, which contradicts the assumption of the lemma that the closure of $K$ contains a neutral fixed point of $g$.
\end{proof}

\section{Proof of Theorem~\ref{thm: theorem A}}


\label{sec: The main proof}

In this section we will prove our first main result, which stated that for any closed $K_1 \subseteq \mathcal{U}_\infty$ and any open $K_2 \supseteq \overline{\mathcal{U}_\infty}$
there exists a $d_0$ such that $K_1 \subseteq d \cdot \mathcal{U}_d \subseteq K_2$ for $d \geq d_0$.

The strategy to prove that $K_1/d$ is eventually zero-free is to show that for any $\Lambda \in \mathcal{U}_\infty$ there exists a convex set $K_{\Lambda}$ that is strictly invariant for $E_{\Lambda}$, i.e. the closure of $E_{\Lambda}(K)$ is contained in the interior of $K$. This allows us to apply Lemma~\ref{lem: strictly forward invariant} to find that there is a neighborhood $X_{\Lambda}$ of $\Lambda$ such that $X_{\Lambda}/d$ is eventually zero-free. Because $K_1$ is compact it follows that $K_1/d$ is eventually zero-free.

The strategy to prove that $K_2/d$ eventually contains $\mathcal{U}_d$ is to show that for $\Lambda \not \in K_2$ we can find elements $g_{\Lambda} \in G_{\Lambda}$ such that $|g_{\Lambda}(0)|$ is very large. We can then use Lemma~\ref{lem: approximate ratios} to find a sequence of ratios $R_d$ of finite graphs with degree at most $d$ and root degree at most $\frac{d}{2}$ such that $|R_{d}(\Lambda/d)|$ is very large compared to the cardioid $\mathcal{C}_{\lfloor \frac{d}{2} \rfloor}$. Because zeros are dense outside the cardioid this will allow us to conclude, using Lemma~\ref{lem: implementation}, that a neighborhood of $\Lambda$ is eventually contained in the rescaled set of zeros $d \cdot \overline{\mathcal{Z}_d}$.


\subsection{Strictly forward invariant regions for nonreal $\Lambda \in \mathcal{U}_\infty$} 

\begin{lemma}
\label{lem: convex}
Let $\Lambda \in \mathbb C \setminus \mathbb R$, and suppose that $t\cdot \Lambda \in \mathcal{V}_\infty$ for some $t > 1$. Then there exists a convex bounded set $K \subset \mathbb C$ with $0 \in K$, which is strictly invariant under the map $E_\Lambda$.
\end{lemma}

\begin{proof}
Let $t' = \frac{1+t}{2}$, and define
$$
{\bf{V}} := (1/t') \cdot \widehat{V}_{t\cdot \Lambda}.
$$
Observe that $\widehat{V}_{\Lambda} \subset {\bf V} \subset \widehat{V}_{t\cdot \Lambda}$.

We further remark that
\begin{equation}
E_\Lambda({\bf V}) \subset E_\Lambda (\widehat{V}_{t\cdot \Lambda}) \subset \frac{1}{t} \widehat{V}_{t\cdot \Lambda} = \frac{t'}{t} {\bf V}.
\end{equation}
It follows that $E_\Lambda({\bf V})$ can intersect ${\bf V}$ only at points lying on radial arcs contained in the boundary of $\widehat{V}_{t\cdot \Lambda}$. Note that there must be exactly two such radial arcs, see Figure~\ref{fig:example} for an example. Indeed, $\widehat{V}_{t\cdot \Lambda} \subset \mathbb C$ is bounded, hence its closure contains points $m$ and $M$ of minimal and maximal imaginary parts. Their images under $E_\Lambda$ must have maximal and minimal arguments respectively, and are nonzero. It follows that the radial arcs through $E_\Lambda(m)$ and $E_\Lambda(M)$ are contained in the boundary of $\widehat{V}_{t\cdot \Lambda}$. We will denote these two radial arcs by $I_1$ and $I_2$.

If $E_\Lambda({\bf V})$ does not intersect the boundary of ${\bf V}$ then choosing $K = \bf V$ completes the proof. We may therefore assume that there exists $p \in {\bf V}$ with $E_\Lambda(p) = q \in I_1$, say. Since $E_\Lambda$ is an open map, it must follow that $p \in I_1\cup I_2$ as well, for if $p$ had been an interior point of $\Vhat_{t \cdot \Lambda}$ the forward invariant set $\Vhat_{t \cdot \Lambda}$ would have been larger.

Now consider the image $E_\Lambda(r \cdot p)$ for  $1-\epsilon < r < 1+\epsilon$ and $\epsilon>0$ small. Assume first that the point $p$ is not contained in the real axis. In that case the curve obtained by varying $r$ is an exponential spiral, intersecting the real line through $q$ transversally. Since ${\bf V}$ is invariant and star-shaped from the origin we obtain a contradiction. Note that we still obtain a contradiction when $p \in I_1 \cap {\bf V}$ is the point of maximal modulus, since we can consider $r\cdot p\in \widehat{V}_{t\cdot \Lambda}$. It follows that $p$ must be real.

There are now two possibilities. The first possibility is that one interval (say $I_2$) is contained in the real line, and that the other interval $I_1$ is contained in $\Lambda \cdot \mathbb R_+ $. In that case $E_\Lambda(I_2)$ is strictly contained in $I_1 \subset \Lambda \cdot \mathbb R_+ $. The second possibility is that only the point $0$ is mapped into $I_1$. In both cases $E_\Lambda({\bf V})$ does not intersect $I_2$.

In both cases we can increase $\bf{V}$ by adding a tiny sector adjacent to $I_1$. To be more precise, let $r \cdot e^{i\alpha}$ be the point in $I_1 \cap {\bf V}$ with maximal modulus. We define $K$ by taking the convex hull of ${\bf V}$ and the point $r e^{i(\alpha + \epsilon)}$, for $\epsilon>0$ sufficiently small. As $\epsilon \rightarrow 0$, the set $K$ contains a neighborhood of $E_\Lambda(I_2) \cap I_1$ of minimal radius comparable to $\epsilon$. On the other hand, for small $\epsilon>0$ the image $E(K) \setminus K$ is contained in a disk centered at $\Lambda$ whose radius is of order $o(\epsilon)$, since we can consider the image of an arbitrarily small neighborhood of the origin. It follows that when $\epsilon$ is chosen sufficiently small, the set $E(K)$ is strictly contained in $K$, which completes the proof.
%
\end{proof}

\begin{corollary}
    \label{cor: strict invariant nonreal}
    For $\Lambda \in \mathcal{U}_\infty \setminus \mathbb{R}$ there is 
    an bounded convex set $K_{\Lambda}$ containing $0$ that is 
    strictly invariant under the map $E_{\Lambda}$.
\end{corollary}

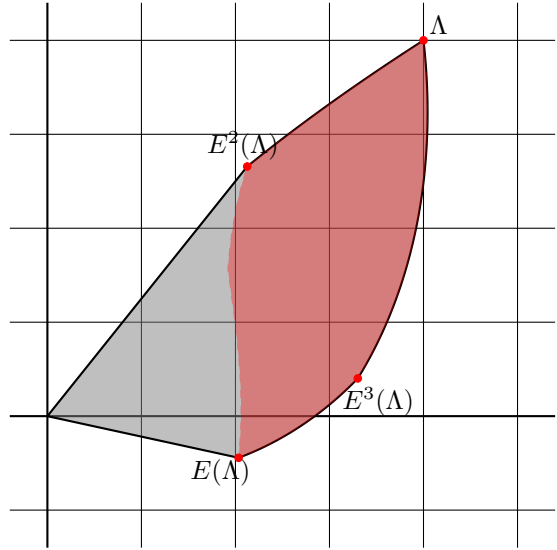
\begin{figure}[ht]

\centering

\begin{tikzpicture}

\draw[thin] (-0.5,-1.25)--(6.75,-1.25);
\draw[thick] (-0.5,0.)--(6.75,0.);
\draw[thin] (-0.5,1.25)--(6.75,1.25);
\draw[thin] (-0.5,2.5)--(6.75,2.5);
\draw[thin] (-0.5,3.75)--(6.75,3.75);
\draw[thin] (-0.5,5.)--(6.75,5.);
\draw[thick] (0.,-1.75)--(0.,5.5);
\draw[thin] (1.25,-1.75)--(1.25,5.5);
\draw[thin] (2.5,-1.75)--(2.5,5.5);
\draw[thin] (3.75,-1.75)--(3.75,5.5);
\draw[thin] (5.,-1.75)--(5.,5.5);
\draw[thin] (6.25,-1.75)--(6.25,5.5);

\filldraw[thick, fill=gray, fill opacity = 0.5] (0.,0.)--(2.65654,3.3216)--(2.65706,3.32203)--(2.65759,3.32246)--(2.65865,3.32331)--(2.66076,3.32502)--(2.66498,3.32845)--(2.67345,3.33531)--(2.69046,3.34906)--(2.7277,3.37907)--(2.7629,3.4073)--(2.79781,3.43517)--(2.83615,3.46564)--(2.87236,3.49429)--(2.91209,3.52557)--(2.9516,3.55652)--(2.98889,3.58561)--(3.02982,3.61739)--(3.0685,3.64727)--(3.10687,3.67679)--(3.149,3.70904)--(3.18879,3.73936)--(3.23244,3.77248)--(3.27584,3.80523)--(3.3168,3.83601)--(3.36176,3.86963)--(3.40424,3.90124)--(3.45083,3.93575)--(3.49715,3.96989)--(3.54088,4.00198)--(3.58887,4.03702)--(3.63422,4.06997)--(3.6792,4.1025)--(3.72858,4.13805)--(3.77521,4.17147)--(3.82637,4.20795)--(3.87722,4.24403)--(3.92522,4.27792)--(3.9779,4.31494)--(4.02767,4.34975)--(4.07702,4.38411)--(4.13121,4.42165)--(4.18237,4.45693)--(4.23851,4.49545)--(4.29155,4.53167)--(4.34416,4.56742)--(4.40191,4.60648)--(4.45645,4.64319)--(4.51627,4.68325)--(4.57573,4.72287)--(4.63185,4.76009)--(4.69343,4.80072)--(4.75161,4.83892)--(4.8093,4.87662)--(4.87263,4.9178)--(4.93243,4.9565)--(4.93348,4.95718)--(4.93453,4.95785)--(4.93663,4.95921)--(4.94083,4.96192)--(4.94924,4.96735)--(4.96611,4.97821)--(4.96717,4.97889)--(4.96822,4.97957)--(4.97033,4.98093)--(4.97456,4.98365)--(4.98303,4.98909)--(4.98409,4.98978)--(4.98515,4.99046)--(4.98727,4.99182)--(4.99151,4.99454)--(4.99257,4.99523)--(4.99363,4.99591)--(4.99575,4.99727)--(4.99681,4.99795)--(4.99788,4.99864)--(4.99894,4.99932)--(5.,5.)--(5.0002,4.99817)--(5.00041,4.99633)--(5.00081,4.99267)--(5.00162,4.98534)--(5.00321,4.9707)--(5.00632,4.94146)--(5.01223,4.88319)--(5.02363,4.75775)--(5.02379,4.75593)--(5.02394,4.75411)--(5.02424,4.75047)--(5.02484,4.74319)--(5.02603,4.72865)--(5.02831,4.69962)--(5.03259,4.64176)--(5.03271,4.64)--(5.03284,4.63823)--(5.03308,4.63469)--(5.03357,4.62763)--(5.03453,4.6135)--(5.03638,4.58531)--(5.0365,4.58355)--(5.03661,4.58179)--(5.03683,4.57828)--(5.03727,4.57124)--(5.03814,4.55719)--(5.0398,4.52913)--(5.03991,4.52723)--(5.04002,4.52533)--(5.04023,4.52153)--(5.04066,4.51394)--(5.04149,4.49877)--(5.04308,4.46849)--(5.04317,4.4666)--(5.04327,4.46471)--(5.04345,4.46094)--(5.04383,4.45338)--(5.04455,4.43829)--(5.04591,4.40817)--(5.04599,4.40642)--(5.04606,4.40466)--(5.04621,4.40116)--(5.04651,4.39414)--(5.04708,4.38013)--(5.04816,4.35216)--(5.04823,4.35042)--(5.04829,4.34867)--(5.04842,4.34518)--(5.04867,4.33821)--(5.04915,4.32427)--(5.05004,4.29644)--(5.05009,4.29455)--(5.05015,4.29267)--(5.05026,4.28891)--(5.05048,4.28139)--(5.05089,4.26636)--(5.05094,4.26448)--(5.05099,4.2626)--(5.05109,4.25885)--(5.05128,4.25135)--(5.05165,4.23636)--(5.05169,4.23449)--(5.05173,4.23262)--(5.05182,4.22887)--(5.05198,4.22139)--(5.05229,4.20644)--(5.05232,4.20458)--(5.05236,4.20271)--(5.05243,4.19898)--(5.05257,4.19152)--(5.05282,4.17661)--(5.05285,4.17478)--(5.05288,4.17296)--(5.05294,4.1693)--(5.05305,4.162)--(5.05307,4.16017)--(5.0531,4.15835)--(5.05315,4.1547)--(5.05324,4.1474)--(5.05327,4.14558)--(5.05329,4.14376)--(5.05333,4.14011)--(5.05341,4.13283)--(5.05343,4.13101)--(5.05345,4.12919)--(5.05349,4.12555)--(5.05356,4.11828)--(5.05358,4.11646)--(5.05359,4.11464)--(5.05362,4.11101)--(5.05364,4.10919)--(5.05365,4.10737)--(5.05368,4.10374)--(5.05369,4.10193)--(5.05371,4.10011)--(5.05373,4.09648)--(5.05378,4.08923)--(5.05379,4.08742)--(5.0538,4.0856)--(5.05381,4.08198)--(5.05382,4.08017)--(5.05383,4.07836)--(5.05385,4.07474)--(5.05385,4.07293)--(5.05386,4.07112)--(5.05387,4.06931)--(5.05387,4.0675)--(5.05388,4.06569)--(5.05388,4.06388)--(5.05389,4.06207)--(5.05389,4.06026)--(5.05389,4.05858)--(5.0539,4.05689)--(5.0539,4.0552)--(5.0539,4.05352)--(5.05391,4.05183)--(5.05391,4.05015)--(5.05391,4.04846)--(5.05391,4.04678)--(5.05391,4.0451)--(5.05391,4.04341)--(5.05391,4.04173)--(5.05391,4.04004)--(5.05391,4.03836)--(5.05391,4.03668)--(5.05391,4.035)--(5.05391,4.03331)--(5.0539,4.03163)--(5.0539,4.02995)--(5.0539,4.02827)--(5.0539,4.02659)--(5.05389,4.02491)--(5.05389,4.02323)--(5.05389,4.02155)--(5.05388,4.01987)--(5.05388,4.01819)--(5.05387,4.01651)--(5.05386,4.01315)--(5.05385,4.01147)--(5.05385,4.00979)--(5.05383,4.00644)--(5.05383,4.00476)--(5.05382,4.00308)--(5.0538,3.99973)--(5.05379,3.99805)--(5.05379,3.99638)--(5.05377,3.99302)--(5.05376,3.99135)--(5.05375,3.98967)--(5.05372,3.98632)--(5.05368,3.97963)--(5.05366,3.97796)--(5.05365,3.97628)--(5.05362,3.97294)--(5.05356,3.96625)--(5.05355,3.96458)--(5.05353,3.96291)--(5.0535,3.95957)--(5.05343,3.95289)--(5.05341,3.95108)--(5.05339,3.94927)--(5.05335,3.94566)--(5.05326,3.93842)--(5.05324,3.93662)--(5.05322,3.93481)--(5.05317,3.9312)--(5.05307,3.92398)--(5.05304,3.92217)--(5.05301,3.92037)--(5.05296,3.91676)--(5.05285,3.90955)--(5.0526,3.89515)--(5.05257,3.89335)--(5.05254,3.89155)--(5.05247,3.88795)--(5.05233,3.88076)--(5.05204,3.8664)--(5.05138,3.83774)--(5.05133,3.83607)--(5.05129,3.8344)--(5.05121,3.83106)--(5.05103,3.82438)--(5.05067,3.81105)--(5.04987,3.78444)--(5.04982,3.78278)--(5.04977,3.78112)--(5.04966,3.7778)--(5.04944,3.77116)--(5.04899,3.7579)--(5.04803,3.73143)--(5.04797,3.72981)--(5.04791,3.72819)--(5.04778,3.72496)--(5.04753,3.71849)--(5.047,3.70556)--(5.0459,3.67975)--(5.04582,3.67814)--(5.04575,3.67653)--(5.04561,3.67331)--(5.04531,3.66687)--(5.04471,3.65402)--(5.04344,3.62835)--(5.04335,3.62661)--(5.04326,3.62488)--(5.04308,3.6214)--(5.04272,3.61446)--(5.04197,3.60058)--(5.04042,3.5729)--(5.03703,3.51778)--(5.03692,3.51618)--(5.03682,3.51458)--(5.03661,3.51137)--(5.03618,3.50497)--(5.03532,3.49218)--(5.03353,3.46665)--(5.02972,3.41581)--(5.02041,3.30659)--(5.00988,3.2007)--(4.99885,3.10314)--(4.9856,2.99864)--(4.97206,2.90233)--(4.95614,2.79931)--(4.93928,2.69953)--(4.92247,2.60766)--(4.90312,2.50938)--(4.884,2.41887)--(4.86431,2.33131)--(4.84191,2.23764)--(4.82007,2.15145)--(4.79539,2.05938)--(4.77018,1.97034)--(4.74582,1.88851)--(4.71849,1.80108)--(4.69216,1.7207)--(4.66561,1.64306)--(4.63598,1.5601)--(4.60761,1.48391)--(4.57608,1.40263)--(4.54592,1.32796)--(4.51573,1.25591)--(4.48228,1.17903)--(4.45045,1.10849)--(4.41529,1.03335)--(4.38014,0.960886)--(4.34681,0.894461)--(4.31008,0.823704)--(4.27529,0.758831)--(4.24071,0.696339)--(4.2027,0.62977)--(4.16678,0.568803)--(4.16615,0.567749)--(4.16552,0.566696)--(4.16426,0.56459)--(4.16173,0.560383)--(4.15668,0.551985)--(4.14655,0.535254)--(4.14592,0.534211)--(4.14528,0.533169)--(4.14401,0.531085)--(4.14147,0.52692)--(4.13639,0.518608)--(4.13575,0.517571)--(4.13512,0.516534)--(4.13384,0.51446)--(4.1313,0.510317)--(4.13066,0.509283)--(4.13002,0.508248)--(4.12875,0.50618)--(4.12811,0.505147)--(4.12747,0.504113)--(4.12684,0.503081)--(4.1262,0.502048)--(4.12581,0.501642)--(4.12541,0.501237)--(4.12463,0.500425)--(4.12306,0.498803)--(4.11992,0.495559)--(4.11362,0.489077)--(4.10097,0.476134)--(4.07331,0.448168)--(4.04718,0.422182)--(4.02129,0.396831)--(3.9929,0.369476)--(3.96612,0.344091)--(3.93679,0.316743)--(3.90768,0.290066)--(3.88026,0.265342)--(3.85021,0.238714)--(3.82189,0.214032)--(3.79386,0.190007)--(3.76317,0.164139)--(3.73426,0.140195)--(3.70264,0.114458)--(3.6713,0.0894189)--(3.64183,0.0662767)--(3.60959,0.041412)--(3.57924,0.0184312)--(3.54608,-0.00621503)--(3.51326,-0.0301472)--(3.4824,-0.0522302)--(3.44869,-0.0759002)--(3.41699,-0.0977378)--(3.38568,-0.118895)--(3.35149,-0.141558)--(3.31936,-0.162431)--(3.28431,-0.184744)--(3.24967,-0.20633)--(3.21717,-0.226173)--(3.18171,-0.247367)--(3.14841,-0.266841)--(3.1156,-0.285635)--(3.07981,-0.30569)--(3.04624,-0.324079)--(3.00968,-0.343655)--(2.97537,-0.361593)--(2.9416,-0.378854)--(2.90481,-0.397206)--(2.87033,-0.413981)--(2.83283,-0.431769)--(2.79587,-0.448833)--(2.76128,-0.464388)--(2.72367,-0.480854)--(2.68845,-0.495845)--(2.65384,-0.510181)--(2.61619,-0.525327)--(2.581,-0.539072)--(2.58038,-0.539308)--(2.57977,-0.539544)--(2.57854,-0.540017)--(2.57608,-0.54096)--(2.57116,-0.54284)--(2.56132,-0.546579)--(2.56071,-0.546812)--(2.56009,-0.547045)--(2.55886,-0.54751)--(2.5564,-0.548438)--(2.55148,-0.550289)--(2.55086,-0.55052)--(2.55025,-0.550751)--(2.54902,-0.551212)--(2.54655,-0.552133)--(2.54594,-0.552363)--(2.54532,-0.552592)--(2.54409,-0.553052)--(2.54348,-0.553281)--(2.54286,-0.55351)--(2.54225,-0.55374)--(2.54163,-0.553969)--(2.54163,-0.553969)--(0.,0.);

\filldraw[red, dashed, opacity = 0.35] (5.,5.)--(5.00644,4.94036)--(5.01244,4.88099)--(5.01802,4.82189)--(5.02319,4.76308)--(5.02793,4.70454)--(5.03227,4.64628)--(5.03619,4.58831)--(5.03971,4.53061)--(5.04284,4.47321)--(5.04556,4.41608)--(5.0479,4.35924)--(5.04984,4.30269)--(5.05141,4.24643)--(5.05259,4.19045)--(5.05339,4.13477)--(5.05383,4.07938)--(5.05389,4.02427)--(5.05359,3.96947)--(5.05293,3.91495)--(5.05192,3.86073)--(5.05055,3.8068)--(5.04883,3.75317)--(5.04676,3.69984)--(5.04436,3.64681)--(5.04162,3.59407)--(5.03854,3.54163)--(5.03513,3.48949)--(5.0314,3.43765)--(5.02734,3.38611)--(5.02297,3.33488)--(5.01828,3.28394)--(5.01328,3.2333)--(5.00797,3.18297)--(5.00235,3.13294)--(4.99644,3.08322)--(4.99023,3.0338)--(4.98372,2.98468)--(4.97693,2.93587)--(4.96984,2.88736)--(4.96248,2.83915)--(4.95484,2.79125)--(4.94692,2.74366)--(4.93872,2.69637)--(4.93026,2.64939)--(4.92154,2.60272)--(4.91255,2.55634)--(4.9033,2.51028)--(4.8938,2.46452)--(4.88404,2.41907)--(4.87404,2.37393)--(4.86379,2.32909)--(4.8533,2.28455)--(4.84258,2.24033)--(4.83161,2.19641)--(4.82042,2.15279)--(4.80899,2.10949)--(4.79735,2.06648)--(4.78547,2.02379)--(4.77338,1.9814)--(4.76108,1.93931)--(4.74856,1.89753)--(4.73583,1.85606)--(4.7229,1.81489)--(4.70976,1.77402)--(4.69642,1.73346)--(4.68289,1.6932)--(4.66916,1.65325)--(4.65524,1.6136)--(4.64113,1.57425)--(4.62684,1.53521)--(4.61236,1.49646)--(4.59771,1.45802)--(4.58288,1.41988)--(4.56788,1.38205)--(4.55271,1.34451)--(4.53737,1.30727)--(4.52186,1.27033)--(4.50619,1.23369)--(4.49037,1.19735)--(4.47439,1.16131)--(4.45825,1.12556)--(4.44197,1.09011)--(4.42554,1.05496)--(4.40896,1.0201)--(4.39224,0.98554)--(4.37538,0.951272)--(4.35839,0.917298)--(4.34126,0.883618)--(4.324,0.850229)--(4.30661,0.817132)--(4.28909,0.784327)--(4.27145,0.751811)--(4.25369,0.719585)--(4.23582,0.687648)--(4.21782,0.655999)--(4.19971,0.624637)--(4.18149,0.593562)--(4.16317,0.562773)--(4.14473,0.532268)--(4.1262,0.502048)--(4.12581,0.501642)--(4.12541,0.501237)--(4.12463,0.500425)--(4.12306,0.498803)--(4.11992,0.495559)--(4.11362,0.489077)--(4.10097,0.476134)--(4.07331,0.448168)--(4.04718,0.422182)--(4.02129,0.396831)--(3.9929,0.369476)--(3.96612,0.344091)--(3.93679,0.316743)--(3.90768,0.290066)--(3.88026,0.265342)--(3.85021,0.238714)--(3.82189,0.214032)--(3.79386,0.190007)--(3.76317,0.164139)--(3.73426,0.140195)--(3.70264,0.114458)--(3.6713,0.0894189)--(3.64183,0.0662767)--(3.60959,0.041412)--(3.57924,0.0184312)--(3.54608,-0.00621503)--(3.51326,-0.0301472)--(3.4824,-0.0522302)--(3.44869,-0.0759002)--(3.41699,-0.0977378)--(3.38568,-0.118895)--(3.35149,-0.141558)--(3.31936,-0.162431)--(3.28431,-0.184744)--(3.24967,-0.20633)--(3.21717,-0.226173)--(3.18171,-0.247367)--(3.14841,-0.266841)--(3.1156,-0.285635)--(3.07981,-0.30569)--(3.04624,-0.324079)--(3.00968,-0.343655)--(2.97537,-0.361593)--(2.9416,-0.378854)--(2.90481,-0.397206)--(2.87033,-0.413981)--(2.83283,-0.431769)--(2.79587,-0.448833)--(2.76128,-0.464388)--(2.72367,-0.480854)--(2.68845,-0.495845)--(2.65384,-0.510181)--(2.61619,-0.525327)--(2.581,-0.539072)--(2.58038,-0.539308)--(2.57977,-0.539544)--(2.57854,-0.540017)--(2.57608,-0.54096)--(2.57116,-0.54284)--(2.56132,-0.546579)--(2.56071,-0.546812)--(2.56009,-0.547045)--(2.55886,-0.54751)--(2.5564,-0.548438)--(2.55148,-0.550289)--(2.55086,-0.55052)--(2.55025,-0.550751)--(2.54902,-0.551212)--(2.54655,-0.552133)--(2.54594,-0.552363)--(2.54532,-0.552592)--(2.54409,-0.553052)--(2.54348,-0.553281)--(2.54286,-0.55351)--(2.54225,-0.55374)--(2.54163,-0.553969)--(2.54173,-0.553014)--(2.54183,-0.55206)--(2.54203,-0.550152)--(2.54242,-0.546338)--(2.5432,-0.538718)--(2.54472,-0.523512)--(2.54764,-0.493236)--(2.55338,-0.428195)--(2.55346,-0.427252)--(2.55353,-0.42631)--(2.55369,-0.424425)--(2.554,-0.420657)--(2.55461,-0.41313)--(2.55579,-0.398108)--(2.55804,-0.368195)--(2.55811,-0.367281)--(2.55817,-0.366367)--(2.5583,-0.364541)--(2.55857,-0.360889)--(2.55908,-0.353594)--(2.56009,-0.339034)--(2.56015,-0.338126)--(2.56021,-0.337217)--(2.56033,-0.335401)--(2.56058,-0.331769)--(2.56106,-0.324515)--(2.56199,-0.310036)--(2.56205,-0.309055)--(2.56211,-0.308075)--(2.56223,-0.306116)--(2.56247,-0.302199)--(2.56295,-0.294374)--(2.56388,-0.27876)--(2.56393,-0.277785)--(2.56399,-0.276811)--(2.5641,-0.274863)--(2.56432,-0.27097)--(2.56476,-0.263191)--(2.5656,-0.247669)--(2.56565,-0.246765)--(2.56569,-0.245861)--(2.56579,-0.244053)--(2.56598,-0.240441)--(2.56634,-0.233223)--(2.56706,-0.218817)--(2.5671,-0.217917)--(2.56714,-0.217019)--(2.56723,-0.215221)--(2.5674,-0.211628)--(2.56773,-0.204449)--(2.56837,-0.190121)--(2.56841,-0.189152)--(2.56846,-0.188183)--(2.56854,-0.186245)--(2.5687,-0.182373)--(2.56903,-0.174636)--(2.56907,-0.17367)--(2.56911,-0.172704)--(2.56919,-0.170772)--(2.56934,-0.166911)--(2.56964,-0.159196)--(2.56968,-0.158233)--(2.56972,-0.157269)--(2.56979,-0.155343)--(2.56994,-0.151493)--(2.57022,-0.143801)--(2.57025,-0.14284)--(2.57029,-0.14188)--(2.57036,-0.139959)--(2.57049,-0.13612)--(2.57076,-0.12845)--(2.57079,-0.127509)--(2.57082,-0.126569)--(2.57088,-0.124689)--(2.57101,-0.12093)--(2.57104,-0.119991)--(2.57107,-0.119052)--(2.57113,-0.117174)--(2.57125,-0.113421)--(2.57128,-0.112483)--(2.57131,-0.111545)--(2.57136,-0.10967)--(2.57148,-0.105922)--(2.57151,-0.104986)--(2.57153,-0.104049)--(2.57159,-0.102177)--(2.5717,-0.098434)--(2.57173,-0.0974986)--(2.57175,-0.0965635)--(2.57181,-0.0946937)--(2.57183,-0.0937591)--(2.57186,-0.0928246)--(2.57191,-0.0909561)--(2.57194,-0.0900221)--(2.57197,-0.0890882)--(2.57202,-0.087221)--(2.57212,-0.0834885)--(2.57214,-0.0825558)--(2.57217,-0.0816232)--(2.57222,-0.0797586)--(2.57224,-0.0788265)--(2.57227,-0.0778946)--(2.57231,-0.0760312)--(2.57234,-0.0750998)--(2.57236,-0.0741685)--(2.57238,-0.0732374)--(2.57241,-0.0723064)--(2.57243,-0.0713756)--(2.57245,-0.070445)--(2.57248,-0.0695145)--(2.5725,-0.0685842)--(2.57252,-0.0677166)--(2.57254,-0.0668491)--(2.57256,-0.0659817)--(2.57258,-0.0651144)--(2.5726,-0.0642474)--(2.57262,-0.0633804)--(2.57265,-0.0625136)--(2.57267,-0.0616469)--(2.57269,-0.0607804)--(2.57271,-0.059914)--(2.57273,-0.0590477)--(2.57275,-0.0581816)--(2.57277,-0.0573156)--(2.57279,-0.0564497)--(2.5728,-0.055584)--(2.57282,-0.0547184)--(2.57284,-0.053853)--(2.57286,-0.0529877)--(2.57288,-0.0521225)--(2.5729,-0.0512575)--(2.57292,-0.0503926)--(2.57294,-0.0495279)--(2.57296,-0.0486633)--(2.57297,-0.0477988)--(2.57299,-0.0469345)--(2.57301,-0.0460703)--(2.57305,-0.0443423)--(2.57307,-0.0434785)--(2.57308,-0.0426148)--(2.57312,-0.0408879)--(2.57314,-0.0400247)--(2.57315,-0.0391616)--(2.57319,-0.0374358)--(2.5732,-0.0365731)--(2.57322,-0.0357105)--(2.57325,-0.0339858)--(2.57327,-0.0331237)--(2.57329,-0.0322616)--(2.57332,-0.030538)--(2.57338,-0.0270924)--(2.5734,-0.0262313)--(2.57341,-0.0253704)--(2.57344,-0.023649)--(2.5735,-0.0202077)--(2.57352,-0.0193477)--(2.57353,-0.0184879)--(2.57356,-0.0167686)--(2.57362,-0.0133316)--(2.57363,-0.0124003)--(2.57365,-0.0114691)--(2.57367,-0.00960718)--(2.57373,-0.00588526)--(2.57374,-0.00495517)--(2.57376,-0.00402524)--(2.57379,-0.00216586)--(2.57384,0.00155103)--(2.57385,0.00247986)--(2.57386,0.00340853)--(2.57389,0.0052654)--(2.57393,0.00897726)--(2.57402,0.0163935)--(2.57403,0.0173198)--(2.57405,0.018246)--(2.57407,0.0200979)--(2.57411,0.0237997)--(2.57418,0.0311961)--(2.5743,0.0459591)--(2.5743,0.0468192)--(2.57431,0.0476793)--(2.57432,0.0493989)--(2.57434,0.0528367)--(2.57438,0.0597057)--(2.57444,0.0734185)--(2.57444,0.0742744)--(2.57444,0.0751302)--(2.57444,0.0768414)--(2.57445,0.0802622)--(2.57446,0.0870976)--(2.57446,0.100743)--(2.57446,0.101578)--(2.57446,0.102413)--(2.57446,0.104082)--(2.57445,0.10742)--(2.57443,0.114088)--(2.57438,0.127402)--(2.57438,0.128233)--(2.57437,0.129064)--(2.57436,0.130725)--(2.57435,0.134047)--(2.5743,0.140684)--(2.5742,0.153935)--(2.57419,0.154832)--(2.57418,0.155729)--(2.57417,0.157524)--(2.57413,0.16111)--(2.57406,0.168277)--(2.57389,0.182583)--(2.57347,0.211089)--(2.57345,0.211918)--(2.57344,0.212747)--(2.57341,0.214404)--(2.57335,0.217718)--(2.57323,0.224339)--(2.57297,0.23756)--(2.57238,0.26391)--(2.57079,0.32063)--(2.56883,0.375788)--(2.56666,0.426776)--(2.56396,0.481598)--(2.56114,0.532334)--(2.55776,0.586852)--(2.55414,0.63993)--(2.55051,0.689061)--(2.54631,0.741922)--(2.54217,0.790909)--(2.5379,0.838599)--(2.53307,0.889975)--(2.52837,0.937593)--(2.5231,0.988858)--(2.51776,1.03885)--(2.51263,1.08519)--(2.50693,1.13515)--(2.5015,1.18151)--(2.49607,1.2267)--(2.49009,1.27548)--(2.48442,1.32075)--(2.47821,1.36957)--(2.47234,1.41493)--(2.46655,1.4592)--(2.46022,1.50699)--(2.45429,1.5514)--(2.44785,1.59931)--(2.44152,1.64615)--(2.43562,1.68967)--(2.42925,1.73667)--(2.42332,1.78039)--(2.41754,1.8231)--(2.41131,1.86928)--(2.40556,1.91223)--(2.40546,1.91298)--(2.40536,1.91373)--(2.40516,1.91523)--(2.40476,1.91822)--(2.40396,1.9242)--(2.40237,1.93615)--(2.40227,1.9369)--(2.40218,1.93764)--(2.40198,1.93914)--(2.40158,1.94212)--(2.40079,1.94809)--(2.40069,1.94884)--(2.40059,1.94958)--(2.40039,1.95108)--(2.4,1.95406)--(2.3999,1.9548)--(2.3998,1.95555)--(2.3996,1.95704)--(2.39951,1.95779)--(2.39941,1.95853)--(2.39931,1.95928)--(2.39921,1.96002)--(2.39924,1.96037)--(2.39927,1.96072)--(2.39933,1.96142)--(2.39944,1.96281)--(2.39968,1.9656)--(2.40015,1.97119)--(2.4011,1.98241)--(2.40324,2.00688)--(2.40531,2.02992)--(2.40742,2.05269)--(2.4098,2.07759)--(2.4121,2.10103)--(2.4147,2.12663)--(2.41735,2.15197)--(2.41992,2.1758)--(2.42281,2.20185)--(2.42561,2.22636)--(2.42846,2.25057)--(2.43167,2.27704)--(2.43477,2.30192)--(2.43826,2.32911)--(2.44182,2.35601)--(2.44526,2.38128)--(2.44913,2.4089)--(2.45288,2.43486)--(2.45709,2.4632)--(2.46137,2.49124)--(2.46552,2.51757)--(2.47017,2.54632)--(2.47467,2.57336)--(2.47923,2.60003)--(2.48435,2.62916)--(2.4893,2.65652)--(2.49484,2.68637)--(2.50049,2.71587)--(2.50593,2.74355)--(2.51203,2.77375)--(2.51792,2.80211)--(2.52387,2.83006)--(2.53055,2.86057)--(2.53699,2.88918)--(2.5442,2.92037)--(2.55115,2.94965)--(2.55817,2.97849)--(2.56603,3.00993)--(2.57361,3.03941)--(2.58207,3.0715)--(2.59065,3.10314)--(2.5989,3.13278)--(2.60812,3.16505)--(2.61699,3.19528)--(2.62593,3.22502)--(2.63592,3.25739)--(2.64551,3.28768)--(2.64568,3.28821)--(2.64585,3.28874)--(2.64619,3.2898)--(2.64687,3.29192)--(2.64824,3.29615)--(2.65098,3.30463)--(2.65115,3.30516)--(2.65133,3.30569)--(2.65167,3.30675)--(2.65236,3.30887)--(2.65375,3.31311)--(2.65392,3.31364)--(2.6541,3.31417)--(2.65444,3.31523)--(2.65514,3.31736)--(2.65531,3.31789)--(2.65549,3.31842)--(2.65584,3.31948)--(2.65601,3.32001)--(2.65619,3.32054)--(2.65636,3.32107)--(2.65654,3.3216)--(2.65654,3.3216)--(2.67377,3.33557)--(2.69111,3.34959)--(2.70855,3.36366)--(2.7261,3.37778)--(2.74375,3.39196)--(2.76151,3.40619)--(2.77938,3.42047)--(2.79735,3.4348)--(2.81543,3.44919)--(2.83361,3.46363)--(2.85191,3.47812)--(2.87032,3.49267)--(2.88883,3.50727)--(2.90746,3.52193)--(2.92619,3.53663)--(2.94504,3.5514)--(2.964,3.56621)--(2.98308,3.58109)--(3.00227,3.59601)--(3.02157,3.61099)--(3.04099,3.62603)--(3.06052,3.64112)--(3.08017,3.65626)--(3.09994,3.67146)--(3.11982,3.68672)--(3.13983,3.70203)--(3.15995,3.7174)--(3.18019,3.73282)--(3.20055,3.7483)--(3.22103,3.76384)--(3.24164,3.77943)--(3.26236,3.79508)--(3.28321,3.81078)--(3.30419,3.82655)--(3.32529,3.84236)--(3.34651,3.85824)--(3.36786,3.87417)--(3.38933,3.89016)--(3.41094,3.90621)--(3.43267,3.92232)--(3.45453,3.93848)--(3.47651,3.95471)--(3.49863,3.97099)--(3.52088,3.98732)--(3.54326,4.00372)--(3.56578,4.02018)--(3.58842,4.03669)--(3.6112,4.05326)--(3.63412,4.0699)--(3.65717,4.08659)--(3.68035,4.10334)--(3.70368,4.12015)--(3.72714,4.13702)--(3.75074,4.15395)--(3.77448,4.17094)--(3.79835,4.18799)--(3.82237,4.2051)--(3.84653,4.22227)--(3.87083,4.2395)--(3.89528,4.2568)--(3.91987,4.27415)--(3.9446,4.29156)--(3.96948,4.30904)--(3.99451,4.32657)--(4.01968,4.34417)--(4.045,4.36183)--(4.07047,4.37955)--(4.09609,4.39734)--(4.12185,4.41518)--(4.14777,4.43309)--(4.17385,4.45106)--(4.20007,4.46909)--(4.22645,4.48719)--(4.25298,4.50535)--(4.27967,4.52357)--(4.30652,4.54185)--(4.33352,4.5602)--(4.36068,4.57861)--(4.388,4.59709)--(4.41548,4.61563)--(4.44311,4.63423)--(4.47092,4.65289)--(4.49888,4.67163)--(4.52701,4.69042)--(4.5553,4.70928)--(4.58375,4.7282)--(4.61237,4.74719)--(4.64116,4.76625)--(4.67012,4.78536)--(4.69925,4.80455)--(4.72854,4.8238)--(4.75801,4.84311)--(4.78765,4.86249)--(4.81746,4.88194)--(4.84744,4.90145)--(4.8776,4.92103)--(4.90793,4.94067)--(4.93845,4.96038)--(4.96913,4.98016)--(5.,5.);

\draw[red, fill = red] (5,5) circle (0.05);
\node at (5.2,5.2) {$\Lambda$};

\draw[red, fill = red] (2.54163,-0.553969) circle (0.05);
\node at (2.3,-0.75) {$E(\Lambda)$};

\draw[red, fill = red] (2.65654,3.3216) circle (0.05);
\node at (2.6,3.6) {$E^2(\Lambda)$};

\draw[red, fill = red] (4.1262,0.502048) circle (0.05);
\node at (4.4,0.2) {$E^3(\Lambda)$};

\end{tikzpicture}

\caption{An example of a set $\Vhat_\Lambda$ and the subset formed by its image under $E = E_\Lambda$, for the parameter $\Lambda = 1+i$. In this example the two radial arcs $I_1$ and $I_2$ end at the points $E(\Lambda)$ and $E^2(\Lambda)$.}
\label{fig:example}
\end{figure}

\subsection{Strictly forward invariant regions for real $\Lambda \in \mathcal{U}_\infty$}
When $\Lambda \in (-1/e,1/e)$ it is easy to see that the image of the open 
unit disk $\mathbb{D}$ under the map $E_{\Lambda}$ is relatively compact in $\mathbb{D}$. 
This takes care of all $\Lambda\in \mathcal{U}_\infty \cap \mathbb{R}_{\leq 0}$.
For $\Lambda >0$ we will use the same strategy as that in \cite{PR19}, 
namely to find a a change of coordinates which transforms the map to a strict contraction on the positive real axis. Using the contraction we will be able to construct a convex forward invariant set.

The map $E_{\Lambda}$
sends $\mathbb{R}_{\geq 0}$ to $\mathbb{R}_{\geq 0}$. 
Define the following bijection
of $\mathbb{R}_{\geq 0}$ by
\[
	\phi(x) = \log(1 + x) \quad\text{ with }\quad \phi^{-1}(x) = e^x - 1
\]
and let
\[
	h_{\Lambda}(x) = \left(\phi \circ E_{\Lambda} \circ \phi^{-1}\right)(x).
\]
While it is true that for all $\Lambda \in (0,e)$ the orbit of any
positive real number under the orbit of $E_{\Lambda}$ converges to a real attracting fixed point,
the maps $E_{\Lambda}$ are not all contractions on the positive real line. The maps $h_{\Lambda}$ have this added benefit.

\begin{lemma}
    Let $\Lambda \in (0,e)$. Then $h_{\Lambda}$ is a strict contraction on 
    the positive real line, i.e.\ $|h_{\Lambda}'(x)| < 1$ for all $x\in \mathbb{R}_{\geq 0}$. 
\end{lemma}

\begin{proof}
    We calculate that 
    \[
        h_{\Lambda}'(x) = -\frac{e^x}{1 + \frac{1}{\Lambda}e^{e^x-1}}.
    \]
    Therefore the inequality $|h_{\Lambda}'(x)| < 1$ is equivalent 
    to the inequality $e^x-1 < \frac{1}{\Lambda}e^{e^x-1}$. Note that 
    for all real numbers $y$ we have $y \leq e^{y-1}$ and thus, letting
    $y = e^x-1$, we find 
    \(
        y \leq e^{y-1} < \frac{1}{\Lambda} e^{y},
    \)
    which is what was required.
\end{proof}

\begin{lemma}\label{lemma:positive strict invariant}
    For $\Lambda \in \mathcal{U}_\infty \cap \mathbb{R}$ there is 
    an bounded convex set $K_{\Lambda}$ containing $0$ that is 
    strictly invariant under the map $E_{\Lambda}$.
\end{lemma}

\begin{proof}
    As discussed previously, for $\Lambda \in (-1/e,1/e)$ we can take $K_{\Lambda}$ to be 
    the open unit disk. We may therefore assume that $\Lambda > 0$. The map $\phi$, as defined
    above, can be extended to an holomorphic map on the half plane 
    $\{z: \textrm{Re}(z) > -1\}$. It follows that we can extend 
    $h_{\Lambda}$ to a map on a neighborhood of the positive real line. 
    
    Let $I = \phi([0,\Lambda])$. Because $[0,\Lambda]$ is forward invariant under $E_{\Lambda}$, the interval $I$ is forward invariant under $h_{\Lambda}$.
    It follows from the fact that $h_{\Lambda}$ 
    is a strict contraction on $I$ that for small enough $\epsilon$ the 
    tubular neighborhood
    $I_{\epsilon} = \{z \in \mathbb{C}: d(I,z) < \epsilon\}$
    gets mapped strictly inside itself by $h_{\Lambda}$. That is, 
    $h_{\Lambda}(I_\epsilon)$ is relatively compact in $I_\epsilon$. 
    It follows that
    $E_{\Lambda}(\phi^{-1}(I_\epsilon))$ is relatively compact in 
    $\phi^{-1}(I_\epsilon)$.
    
    All that remains to be shown is that $\phi^{-1}(I_\epsilon)$ is convex for $\epsilon$ small enough. Write $I = [a,b]$ and for a ball 
    of radius $\epsilon$ centered at $z$ write $B_{\epsilon}(z)$,
    then we see that 
    \[
        \exp(I_\epsilon) = \exp\left(\bigcup_{x \in [a,b]} B_{\epsilon}(x)\right)
        =
        \bigcup_{x \in [a,b]} \exp\left(B_{\epsilon}(x)\right)
        =
        \bigcup_{x \in [a,b]} e^{x} \exp(B_{\epsilon}(0))
        = 
        [e^a, e^b] \cdot \exp(B_{\epsilon}(0)).
    \]
    If $\epsilon < 1$ then 
    $\exp(B_\epsilon(0))$ is convex (see e.g.~equation~{(\ref{eq: convexity}})), and thus $\phi(I_\epsilon)$ is convex as an entrywise product of convex sets.
\end{proof}

\subsection{Conclusion of Theorem \ref{thm: theorem A}}
\label{sec: the proof}

Let $K_1$ be a closed set contained in $\mathcal{U}_\infty$. For $\Lambda \in K_1$ let $K_\Lambda$ denote a bounded convex set containing $0$ that is strictly invariant under the map $E_{\Lambda}$. Such sets exist by Lemmas~\ref{lem: convex}~and~\ref{lemma:positive strict invariant}. It follows from Lemma~\ref{lem: strictly forward invariant} that there exists a neighborhood $U_{\Lambda}$ of $\Lambda$ and a $d_{\Lambda}$ such that for $d \geq d_{\Lambda}$ we have that $U_{\Lambda} \subseteq d \cdot \mathcal{U}_d$. The set $K_1$ is compact and can thus be covered by finitely such $U_{\Lambda}$, say $U_{\Lambda_1}, \dots, U_{\Lambda_n}$. By taking $d_0 = \max_{i=1,\dots, n} d_{\Lambda_i}$ we see that $K_1 \subseteq d \cdot \mathcal{U}_\infty$ for $d \geq d_0$.

Let $K_2$ be an open set containing $\overline{\mathcal{U}_\infty}$. We will show that for large enough $d$ the sets $d\cdot\mathcal{U}_d$ are contained in $K_2$. In other words: we will show that the complement $(K_2)^c$ is contained in $d \cdot \overline{\mathcal{Z}_d}$ for $d$ large enough. For this we use that zeros are dense outside the cardioid for any finite degree, which was proved in \cite{BBGPR}, building upon results from \cite{Galanisetal20}:
\begin{lemma}
    \label{lem: zeros outside}
    Let $d \geq 2$, the zero locus $\overline{\mathcal{Z}_d}$ contains the complement of the cardioid, i.e. $(\mathcal{C}_d)^c \subseteq \overline{\mathcal{Z}_d}$.
\end{lemma}

Because the real interval $(-1/e,e)$ is contained in $\mathcal{U}_\infty$ there is a real number $\delta > 1$ such that $\delta \cdot (-1/e,e)$ is contained in 
$K_2$. It follows that we can cover the complement $(K_2)^c$ by the 
union of the complement of $(\delta \cdot \mathcal{C}_\infty)$ and a compact 
set $K$ that does not intersect $\mathbb{R} \cup \mathcal{V}_\infty$.

The sets $d \cdot \mathcal{C}_d$ converge to $\mathcal{C}_\infty$ and thus by Lemma~\ref{lem: zeros outside} it follows that $(\delta \cdot \mathcal{C}_\infty)^c \subseteq d \cdot \overline{\mathcal{Z}_d}$ for large enough $d$. It remains to be shown that we can choose $d$ large enough such that the sets $d \cdot \overline{\mathcal{Z}_d}$ also contain $K$. By compactness of $K$ it is sufficient to show that every $\Lambda \in K$ has a neighborhood $X_{\Lambda}$ that is contained in $d \cdot \overline{\mathcal{Z}_d}$ for sufficiently large degree $d$.

Let $\Lambda_0 \in K$. Because $\Lambda_0$ is neither real nor a member of $\mathcal{V}_\infty$, Lemma~\ref{lem: Vhat is C} implies that $V_{\Lambda_0} = \mathbb{C}$. Let $Z_0 \in \mathbb{C}$ be a value such that $\Lambda_0 e^{-Z_0} = 100$. Choose $g_\Lambda \in G_{\Lambda}$ such that $g_{\Lambda_0}(0) = 2 Z_0$. It follows from Lemma~\ref{lem: approximate ratios} that there is a sequence of rooted graphs $(G_d,v_d) \in \mathcal{G}_d^{\lfloor\frac{d}{2}\rfloor}$ such the maps $\Lambda \mapsto d \cdot R_{G_d,v_d}(\Lambda/d)$ converge uniformly on compact subsets to $\Lambda \mapsto \Lambda \cdot e^{-\frac{1}{2}g_{\Lambda}(0)}$. Therefore there 
is a $d_0$ and a neighborhood $X_{\Lambda_0}$ such that for $d \geq d_0$
and $\Lambda \in X_{\Lambda_0}$ we have $|d \cdot R_{G_d,v_d}(\Lambda/d)| \geq 99$. The cardioid $\mathcal{C}_{\lfloor\frac{d}{2}\rfloor}$ is contained in a disk of radius 
\[
    \frac{(\lfloor\frac{d}{2}\rfloor)^{\lfloor\frac{d}{2}\rfloor}}{(\lfloor\frac{d}{2}\rfloor-1)^{\lfloor\frac{d}{2}\rfloor+1}},
\]
which for $d\geq 4$ is less than $99/d$. This means that the values $R_{G_d,v_d}(\Lambda/d)$ lie in the complement of $\mathcal{C}_{\lfloor\frac{d}{2}\rfloor}$ and thus in $\overline{\mathcal{Z}_{\lfloor\frac{d}{2}\rfloor}}$. It finally follows from Lemma~\ref{lem: implementation} that for $d \geq d_0$ the sets $d \cdot \overline{\mathcal{Z}_d}$ contain $X_{\Lambda_0}$.

\hfill $\square$

\section{Near the real boundary of $\mathcal{U}_\infty$.} \label{section:realboundary}

\label{sec: gamma curve}

\subsection{Around $e$}
For $\theta \in (0,\pi)$ we define $\gamma(\theta)$ as the unique positive real number that solves
\begin{equation}
    \label{eq: defining gamma}
    \gamma(\theta)^2 - \sin(\gamma(\theta))^2 = \theta^2.
\end{equation}
To keep the notation readable we will drop the dependency on $\theta$ and just write $\gamma=\gamma(\theta)$.
We define 
\[
    \hat{\Lambda}(\theta) = \frac{\gamma+\theta}{\sin\gamma} e^{(\gamma-\theta)\cot\gamma} e^{\theta i},
    \quad
    \hat{c}(\theta) = \frac{\gamma-\theta}{\gamma+\theta} e^{2\theta \cot \gamma}
    \quad
    \text { and }
    \quad
    \hat{Z}(\theta) = \frac{\gamma+\theta}{\sin \gamma} e^{-2\theta \cot \gamma}
    e^{-\gamma i}.
\]
We can continuously extend these functions to be defined for $\theta = 0$ by setting
$\gamma(0) = 0$, $\hat{\Lambda}(0) = e$, $\hat{c}(0) = 1$ and $\hat{Z}(0) = 1$.
We define the composition $H_{\Lambda,c} = E_{\Lambda} \circ E_{\Lambda,c}$.
\begin{lemma}
    Let $\theta \in [0,\pi)$
    then $\hat{Z}(\theta)$ is a fixed point of $H_{\hat{\Lambda}(\theta), \hat{c}(\theta)}$ with multiplier $1$.
\end{lemma}

\begin{proof}
We first check that $\hat Z(\theta)$ is a fixed point, and afterwards compute its multiplier.
 
 As $H_{\Lambda, c}( Z)=E_{\Lambda}(E_{\Lambda, c}( Z))=E_{\Lambda}(E_{\Lambda}( c Z))$ we calculate first
 \begin{align*}
     \hat c(\theta) \hat Z (\theta) = \left(\frac{\gamma-\theta}{\gamma+\theta} e^{2\theta \cot \gamma}\right) \cdot \left(\frac{\gamma+\theta}{\sin \gamma} e^{-2\theta \cot \gamma}
    e^{-\gamma i}\right)
                                    =\frac{\gamma-\theta}{\sin\gamma}e^{-\gamma i}.
 \end{align*}
 Therefore its image
 \begin{align*}
     E_{\hat \Lambda,\hat c}(\hat Z)=E_{\hat\Lambda}(\hat c\hat Z)=\hat \Lambda e^{-\hat c\hat Z}
                    =\left(\frac{\gamma+\theta}{\sin\gamma} e^{(\gamma-\theta)\cot\gamma} e^{\theta i}\right)\left(e^{-(\gamma-\theta)\cot(\gamma)}e^{(\gamma-\theta)i}\right)
                    =\frac{\gamma+\theta}{\sin\gamma}e^{\gamma i}
 \end{align*}
 and
 \begin{align*}
     H_{\hat \Lambda,\hat c}(\hat Z)=\hat\Lambda e^{-E_{\hat \Lambda,\hat c}(\hat Z)}
                    =\left(\frac{\gamma+\theta}{\sin\gamma} e^{(\gamma-\theta)\cot\gamma} e^{\theta i}\right)\left(e^{-(\gamma+\theta)\cot\gamma}e^{-(\gamma+\theta)i}\right)
                    =\frac{\gamma+\theta}{\sin\gamma}e^{-2\theta\cot \gamma}e^{-\gamma i}=\hat Z.
 \end{align*}
 We see that $\hat Z$ is indeed a fixed point. The derivative of $H_{\Lambda,c}(Z)$ is 
 \[
    H_{\Lambda, c}'(Z)=H_{\Lambda, c}(Z)\left(-E_{\Lambda,c}(Z)\right)(-c),
 \]
 and thus
 \begin{align*}
     H_{\hat \Lambda, \hat  c}'(\hat Z)=\hat c E_{\hat \Lambda,\hat c}(\hat Z)H_{\hat \Lambda,\hat c}(\hat Z)
     = E_{\hat \Lambda,\hat c}(\hat Z)\cdot \left( \hat c \hat Z\right)
     = \left(\frac{\gamma+\theta}{\sin\gamma}e^{\gamma i} \right)\left( \frac{\gamma-\theta}{\sin\gamma}e^{-\gamma i} \right)
     =\frac{\gamma^2-\theta^2}{\sin^2\gamma}=1,
 \end{align*}
 which completes the proof.
\end{proof}

\begin{lemma}
    If $\theta \in (0,\pi)$ then $\hat{c}(\theta) \in (0,1)$. 
\end{lemma}

\begin{proof}
Clearly for $\theta \in (0, \pi)$ the quantity $\gamma$ is strictly larger than 
$\theta$, and hence $\hat{c}(\theta)$ is strictly positive.
Because $\gamma \in (0,\pi)$ the inequality $\cot(\gamma) < 1/\gamma$ is valid.
Therefore we can write 
\[
    \hat{c}(\theta) = \frac{\gamma-\theta}{\gamma+\theta} e^{2\theta \cot \gamma} 
    < \frac{\gamma-\theta}{\gamma+\theta} e^{2\theta/\gamma} = \frac{1-x}{1+x} \cdot e^{2x},
\]
where $x = \theta/\gamma$. The right-hand side of this equation describes 
a strictly decreasing function starting at $1$ for $x=0$, as one can check by calculating its derivative. It therefore follows that $\hat{c}(\theta) < 1$.
\end{proof}

From now on we assume that $\theta$ is sufficiently small such that both
$\theta$ and $\gamma$ lie in $(0, \pi/2)$, i.e.\ such that the 
complex numbers $e^{i\theta}$ and $e^{-i\gamma}$ lie in the 
first and fourth quadrant of the plane respectively. We define the set  $T_{\theta} \subsetneq \mathbb{C}$ as the closed convex set bounded by the
following two line segments
\[
    I_1 = \left\{t\cdot(\theta + \gamma) i : t \in [0,1]\right\},
    \quad
    I_2 = \left\{t \cdot \frac{\pi/2-\theta}{\sin \gamma}  e^{-i \gamma} : t \in [0,1]\right\}
\]
and two infinite rays parallel to the real line starting at the endpoints 
of these segments
\[
    I_3 = \left\{(\theta + \gamma) i + t: t \in [0,\infty) \right\},
    \quad 
    I_4 = \left\{\frac{\pi/2-\theta}{\sin \gamma}  e^{-i \gamma} + t: t \in [0,\infty)\right\}.
\]
An explicit example of the set $T_\theta$ is illustrated in Figure~\ref{fig:T}.

\begin{lemma}
    \label{lem: T forward invariant}
    For $\theta$ sufficiently small the region $T_\theta$ contains 
    $\hat{Z}(\theta)$ and is forward invariant for $E_{\hat{\Lambda}(\theta)}$.
\end{lemma}

\begin{proof}
    As $\theta$ approaches $0$ the modulus of $\hat{Z}(\theta)$ goes 
    to $1$, while the length of the line segment $I_2$ goes to infinity. 
    It follows that for $\theta$ sufficiently small $\hat{Z}(\theta) \in I_2$,
    and hence $\hat{Z}(\theta) \in T_\theta$.
    
    To show that $T_{\theta}$ is forward invariant for $E_{\hat{\Lambda}(\theta)}$
    it is sufficient to show that its boundary gets mapped into  $T_{\theta}$, because 
    $E_{\hat{\Lambda}(\theta)}$ is holomorphic. The image of $I_1$ is 
    a circular arc of the form $\{|\hat{\Lambda}(\theta)|e^{\phi i}: \phi \in [-\gamma, \theta]\}$. One of the endpoints of this arc is $\Lambda$ and the 
    other has the same argument as the interval $I_2$. Therefore $E_{\hat{\Lambda}(\theta)}(I_1)$ 
    lies in 
    $T_\theta$ if the imaginary part of $\hat{\Lambda}(\theta)$ 
    is less than or equal to 
    $(\theta + \gamma)$ and the modulus $|\hat{\Lambda}(\theta)|$ is at most the 
    length of $I_2$. The two inequalities involved are
    \begin{align}\label{eq:Gamma bound}
        \frac{\sin \theta }{\sin\gamma} e^{(\gamma-\theta)\cot\gamma} \leq 1
        \quad\text{ and }
        \quad
        (\gamma + \theta) e^{(\gamma-\theta)\cot\gamma} \leq \pi/2 - \theta.
    \end{align}
    As $\theta$ approaches $0$ both left-hand sides approach $0$ and thus 
    the inequalities are satisfied for $\theta$ sufficiently small.
    
    The image of $I_2$ is a part of a logarithmic spiral of the form 
    \[
        \{\hat \Lambda e^{t(\theta-\pi/2)\cot\gamma}e^{t(\pi/2-\theta)i} ~|~ t\in[0,1]\}.
    \]
    This curve starts at $\hat\Lambda$ and rotates $\pi/2-\theta> 0$ in anti-clockwise direction, since we assumed that $\theta\in(0,\pi/2)$. In particular this curve is in the upper half plane as $\arg\hat\Lambda +(\pi/2-\theta)=\pi/2<\pi$, moreover the endpoint is a positive multiple of $i$.
    In order to show that $I_2$ is in $T_\theta$, it is enough to show that the maximal imaginary part for points on the curve is at most $(\theta+\gamma)$. By taking derivatives with respect to $t$, the maximal imaginary part occurs when
    \[
        t^*=\frac{\gamma-\theta}{\pi/2-\theta}. 
    \]
    The imaginary part at $t^*$ is 
    \begin{align*}
        |\hat \Lambda|e^{t^*(\theta-\pi/2)\cot \gamma}\sin(\theta+t^*(\pi/2-\theta))&=\frac{\gamma+\theta}{\sin\gamma}e^{(\gamma-\theta)\cot\gamma} e^{(\theta-\gamma)\cot\gamma}\sin(\theta+\gamma-\theta)\\
        &=\frac{\gamma+\theta}{\sin\gamma}\sin\gamma=\gamma+\theta.
    \end{align*}
    
    Now we are left with the images of $I_3$ and $I_4$. Both of them are straight lines connecting $0\in T_\theta$ with one of the endpoints of the image of $I_1$ or $I_2$. As $T_\theta$ is convex and it contains the images of $I_1$ and $I_2$, the images of $I_3$ and $I_4$ are also contained in $T_\theta$.
\end{proof}

\begin{figure}[ht]
\centering
\begin{tikzpicture}

\draw[black] (-0.375,0.) -- (11.,0.);
\node[black] at (11.75,0.){{\textrm{Re}(Z)}};
\draw[black] (0.,-4.25) -- (0.,2.5);\node[black] at (0.,2.875){{\textrm{Im}(Z)}};

\draw[thick,red] (0.,0.) -- (0.,1.86078);
\draw[thick,red] (0.,0.) -- (5.49311,-3.47699);
\draw[thick,red] (0.,1.86078) -- (11.,1.86078);
\draw[thick,red] (5.49311,-3.47699) -- (11.,-3.47699);

\filldraw[dashed, fill=gray, fill opacity = 0.5] (0,0) -- (0.,0.709428) -- (0.0430703,0.773403) -- (0.0939086,0.840532) -- (0.153407,0.910629) -- (0.222526,0.983435) -- (0.302299,1.05861) -- (0.39383,1.13572) -- (0.498296,1.21423) -- (0.616949,1.29347) -- (0.751113,1.37266) -- (0.902183,1.45084) -- (1.07162,1.52691) -- (1.26096,1.59954) -- (1.47178,1.66723) -- (1.70573,1.72823) -- (1.96447,1.78052) -- (2.24972,1.82182) -- (2.5632,1.84952) -- (2.90663,1.8607) -- (3.28172,1.85203) -- (3.6901,1.8198) -- (4.13335,1.75988) -- (4.61293,1.66764) -- (5.13016,1.53797) -- (5.68616,1.3652) -- (6.28182,1.1431) -- (6.28182,1.1431) -- (6.31306,0.955595) -- (6.33871,0.767244) -- (6.35874,0.578213) -- (6.37313,0.388669) -- (6.38188,0.198781) -- (6.38497,0.00871605) -- (6.3824,-0.181356) -- (6.37417,-0.371268) -- (6.36029,-0.56085) -- (6.34078,-0.749935) -- (6.31565,-0.938356) -- (6.28491,-1.12594) -- (6.24861,-1.31254) -- (6.20677,-1.49796) -- (6.15943,-1.68206) -- (6.10663,-1.86467) -- (6.04841,-2.04563) -- (5.98484,-2.22477) -- (5.91596,-2.40194) -- (5.84184,-2.57698) -- (5.76254,-2.74974) -- (5.67813,-2.92006) -- (5.58869,-3.0878) -- (5.49429,-3.25279) -- (5.39503,-3.41491) -- (0,0);

\node[red] at (-0.375,0.93039){$I_1$};
\node[red] at (2.48139,-2.00366){$I_2$};
\node[red] at (8.25,2.23578){$I_3$};
\node[red] at (8.25,-3.85199){$I_4$};

\draw[fill=black] (6.28182,1.1431) circle (0.05);
\node[black] at (6.64516,1.23588){$\hat{\Lambda}$};
\draw[fill=black] (2.93975,1.86078) circle (0.05);
\node[black] at (2.93975,2.23578){{$E_{\hat{\Lambda}}(\hat{c}\hat{Z})$}};
\node[black] at (3.31475,0.375){{$E_{\hat{\Lambda}}(T_{\theta})$}};\draw[fill=black] (1.66459,-1.05364) circle (0.05);
\node[black] at (1.39942,-1.31881){{$\hat{Z}$}};

\end{tikzpicture}
\caption{The region $T_\theta$ and its image under the map $E_{\hat{\Lambda}}$ for 
$\theta = 0.18$.}
\label{fig:T}
\end{figure}
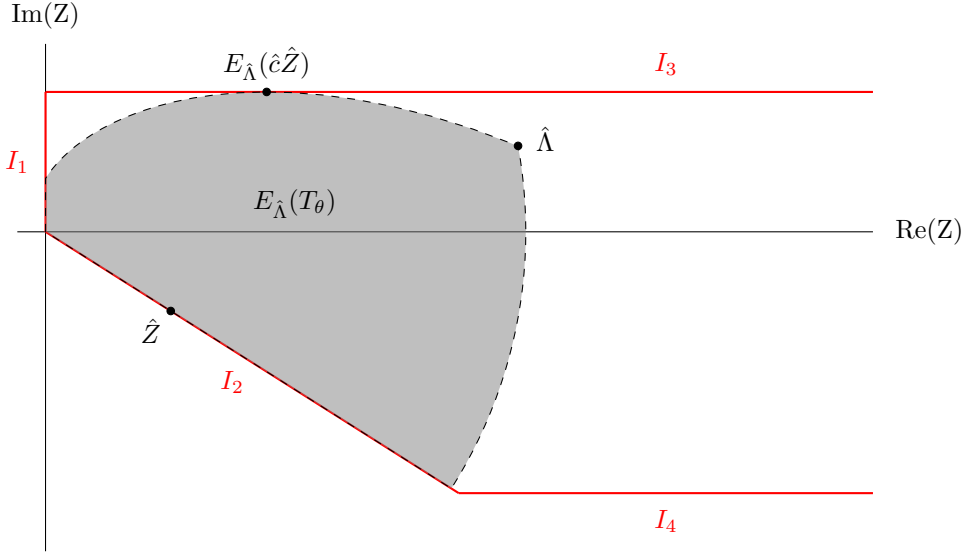

\begin{corollary}
    \label{cor: curve Gamma}
    There exists a $\theta_{\text{max}} > 0$ such that the curves
    \[
        \Gamma = \{\hat{\Lambda}(\theta): \theta \in [0,\theta_{\text{max}}]\}
        \quad \text{ and }\quad
        \overline{\Gamma} = \{\overline{\hat{\Lambda}(\theta)}: \theta \in [0,\theta_{\text{max}}]\}
    \]
    are contained in $\partial \mathcal{U}_\infty$.
\end{corollary}

\begin{proof}
    Without loss of generality it is sufficient to prove that $\Gamma\subseteq \partial\mathcal{U}_\infty$, as the dynamics for $\Lambda$ and $\overline{\Lambda}$ are conjugated. If $\theta=0$, then $\hat\Lambda(0)=e$ and thus $\hat\Lambda(0)\in\partial\mathcal{U}_\infty$. So for the rest we may assume that $\theta>0$, and therefore $\hat\Lambda(\theta)\notin \mathbb{R}$.
    
    By the previous lemma we know that for any sufficiently small $\theta$, the region $T_\theta$ is a convex forward $E_{\hat\Lambda}$-invariant set, hence $\hat\Lambda\in \mathcal{V}_\infty$.
    
    On the other hand, we know that $T_\theta$ contains $\hat Z$, a neutral fixed point of $H_{\hat \Lambda,\hat c}$, therefore by Lemma~\ref{lem:neutral fixed point exclusion} $\hat\Lambda\notin \mathcal{U}_\infty$. In particular, $\hat\Lambda\in\partial \mathcal{U}_\infty$.
\end{proof}

Using the computer it can be seen that the inequalities laid out in Lemma~\ref{lem: T forward invariant} needed for $T_\theta$ to be forward invariant hold up until at least $\theta = 0.18$. This implies that $\theta_{\text{max}}$ is at least $0.18$. In Figure~\ref{fig:cardioid} the curve $\Gamma$ is drawn up until this angle.

\begin{remark} In \cite{Buys21} the second author questioned whether the closure of the zeros of all spherically regular trees equals the closure of the zeros of all graphs, for the same degree bound. We do not resolve this question in this paper, but let us observe that for any $\Lambda_0\in\Gamma$ and any $\varepsilon>0$ the ball $B_{\varepsilon/d}(\Lambda_0/d)$ contains zeros of spherically regular trees for $d$ sufficiently large. 

To see this let us fix a $\Lambda_0\in\Gamma$, and fix the corresponding $c$ for which $H_{\Lambda_0,c}$ has a neutral fixed point $Z$ of multiplier $1$.  Note that $H_{\Lambda,c}$ cannot have a persistently neutral fixed point in the neighborhood of $\Lambda_0$. It follows that in a punctured neighborhood of $\Lambda_0$, the corresponding fixed points are given by a multi-valued holomorphic function, where the number of points equals the multiplicity of the parabolic fixed point. The corresponding multipliers $\phi_i(\Lambda)$ are therefore also given by a multi-valued holomorphic function, and all converge to $1$ as $\Lambda \rightarrow \Lambda_0$. The function $\Phi(\Lambda) = \prod_i (\phi_i(\lambda) - 1)$ therefore extends holomorphically to a full neighborhood of $\Lambda_0$, and has an isolated zero at $\Lambda = \Lambda_0$.

For fixed $d$ define 
\[
h_{\lambda,c} := f_{d,\lambda}\circ f_{\lfloor cd\rfloor,\lambda},
\]
and recall that $d\cdot h_{\Lambda/d,c}(\bullet/d)$ converges locally uniformly to $H_{\Lambda, c}$. It follows from the implicit function theorem that for large $d$ the function $d\cdot h_{\Lambda/d,c}(\bullet/d)$ has fixed points near the fixed points of $H_{\Lambda, c}$, for $\Lambda$ in a small annulus around $\Lambda_0$, whose multipliers are close to the multipliers of $H_{\Lambda,c}$. Hence for large $d$ there exist holomorphic functions $\Phi_d = \prod_i(\phi_{i,d} - 1)$, in terms of the corresponding multipliers $\phi_{i,d}$,  converging to the above function $\Phi$. As above, the functions $\Phi_d$ can be extended to a full neighborhood of $\Lambda_0$.

It follows that for large $d$ the functions $\Phi_d$ must also vanish at some point $\Lambda_d$ in the neighborhood of $\Lambda_0$, meaning $h_{\Lambda/d,c}$ has a neutral fixed point with multiplier $1$ for some $\lambda_d = \Lambda_d/d$  close to the parameter $\Lambda_0/d$.

By Lemma~13 of \cite{Buys21} it follows that $\lambda_d$ lies in the closure of the zeros of spherically regular trees. 

As a corollary, by choosing $\varepsilon>0$ small enough and $d$ large enough, we have a spherically regular tree with a zero in the interior of $C_d$, which resolves an other question of \cite{Buys21} for large degrees.
\end{remark}

\bibliographystyle{amsalpha}
\bibliography{indbib}

\end{document}